\documentclass{amsart}
\ifx\MyBeamer\undefined
\usepackage[breaklinks,colorlinks]{hyperref}
\usepackage[a4paper,margin=1in]{geometry}
\else
\if\MyBeamer t
\usepackage[breaklinks,colorlinks]{hyperref}
\usepackage[a4paper,margin=1in]{geometry}
\else
\def\emph{\alert}
\fi
\fi

\usepackage{amsmath}
\usepackage{amssymb}
\usepackage{mathrsfs}
\usepackage{amsthm}
\usepackage{aliascnt}
\usepackage[utf8]{inputenc}
\usepackage[T1]{fontenc}
\usepackage[nofancy]{svninfo}
\ifx\FrenchText\undefined
\usepackage[british]{babel}
\else
\usepackage[british,french]{babel}
\AtBeginDocument{%
\catcode`\:=12%
\catcode`\!=12%
}
\fi

\makeatletter

\newcounter{quotethmcnt}

\ifx\MyBeamer\undefined

\def\equationautorefname~#1\null{(#1)}
\def\itemautorefname~#1\null{#1}
\fi

\ifx\MyBeamer\undefined

\ifx\thmnum\undefined
\newtheorem{thmnum}{}[section]
\fi

\newcommand{\mynewthm}[3][]{%
  \newaliascnt{#2}{thmnum}%
  \newtheorem{#2}[#2]{#3}%
  \aliascntresetthe{#2}%
  \newtheorem*{#2*}{#3}%
  \expandafter\newcommand\csname #2autorefname\endcsname{#3}%
  \expandafter\renewcommand\csname the#2\endcsname{\thethmnum}%
}

\newtheorem*{clm}{Claim}
\newenvironment{clmprf}{%
  \begin{proof}[Proof of claim]%
  }{\end{proof}}

\else

\let\xxx=\frametitle
\def\frametitle#1{%
  \xxx{%
    \setbeamercolor*{math text}{use={titlelike,my math text},fg=titlelike.fg!80!my math text.fg}%
    #1}%
  \setbeamercolor{math text}{use=my math text,fg=my math text.fg}%
}

\newcommand{\beamerenv}[3]{%
\newenvironment<>{#1}%
{%
  \setbeamercolor{temp}{fg=structure.fg}%
  \setbeamercolor{structure}{fg=#2}%
  \setbeamercolor{block body}{use=structure,bg=structure.fg!5!white}%
  \begin{#3}%
}%
{\end{#3}\setbeamercolor{structure}{fg=temp.fg}}}

\newcommand{\mynewthm}[3][green!50!black]{%
  \newtheorem*{#2x}{#3}%
  \beamerenv{#2}{#1}{#2x}%
}

\beamerenv{other}{violet}{block}

\fi

\ifx\FrenchText\undefined
\newcommand{\myiffrench}[2]{#2}
\else
\newcommand{\myiffrench}[2]{\iflanguage{french}{#1}{#2}}
\fi

\theoremstyle{plain}
\mynewthm[red]{thm}{\myiffrench{Théorème}{Theorem}}
\mynewthm{prp}{Proposition}
\mynewthm[purple]{lem}{\myiffrench{Lemme}{Lemma}}
\mynewthm[orange]{fct}{\myiffrench{Fait}{Fact}}
\mynewthm[purple!50!white]{cor}{\myiffrench{Corollaire}{Corollary}}

\theoremstyle{definition}
\mynewthm[green!80!black]{dfn}{\myiffrench{Définition}{Definition}}
\mynewthm{hyp}{\myiffrench{Hypothèse}{Hypothesis}}
\mynewthm{conv}{Convention}
\mynewthm{conj}{Conjecture}
\mynewthm{ntn}{Notation}
\mynewthm{cst}{Construction}

\theoremstyle{remark}
\mynewthm[yellow!90!black]{rmk}{\myiffrench{Remarque}{Remark}}
\mynewthm{qst}{Question}
\mynewthm{exc}{\myiffrench{Exercice}{Exercise}}
\mynewthm{exm}{\myiffrench{Exemple}{Example}}

\ifx\MyBeamer\undefined

\newcommand{\myenumlabel}[1]{\textnormal{(\roman{#1})}}

\fi

\newcounter{cycprfcnt}
\newcounter{cycprffirst}
\newcommand{\cycprfpreamble}[1]%
{%
  \setcounter{cycprfcnt}{1}
  \setcounter{cycprffirst}{#1}
  \setlength{\itemindent}{0.5\leftmargin}%
  \setlength{\leftmargin}{0pt}%
  \newcommand{\cpcurr}{\myenumlabel{cycprfcnt}}%
  \newcommand{\cpnext}{\addtocounter{cycprfcnt}{1}\cpcurr}%
  \newcommand{\cpprev}{\addtocounter{cycprfcnt}{-1}\cpcurr}%
  \newcommand{\cpnum}[1]{\setcounter{cycprfcnt}{##1}\cpcurr}%
  \newcommand{\cpfirst}{\cpnum{1}}%
  \newcommand{\impnext}{\cpcurr{} $\Longrightarrow$ \cpnext.}%
  \newcommand{\impprev}{\cpcurr{} $\Longrightarrow$ \cpprev.}%
  \newcommand{\impfirst}{\cpcurr{} $\Longrightarrow$ \cpfirst.}%
  \newcommand{\impnum}[2]{\cpnum{##1}{} $\Longrightarrow$ \cpnum{##2}.}%
  \def\makelabel##1{\ifnum\value{cycprffirst}=0\hspace{-0.7\itemindent}\setcounter{cycprffirst}{1}\fi##1}%
}%

\newenvironment{cycprf}[1][0]%
{\begin{list}{\impnext}{\cycprfpreamble{#1}}}%
{\qedhere\end{list}}%

\newenvironment{cycprf*}[1][0]%
{\begin{list}{\impnext}{\cycprfpreamble{#1}}}%
{\end{list}}%

\def\indsym#1#2{%
  \setbox0=\hbox{$\m@th#1x$}%
  \kern\wd0%
  \hbox to 0pt{\hss$\m@th#1\mid$\hbox to 0pt{$\m@th#1^{#2}$\hss}\hss}%
  \lower.9\ht0\hbox to 0pt{\hss$\m@th#1\smile$\hss}%
  \kern\wd0}

\def\nindsym#1#2{%
  \setbox0=\hbox{$\m@th#1x$}%
  \kern\wd0%
  \hbox to 0pt{\hss$\m@th#1\not$\kern1.4\wd0\hss}
  \hbox to 0pt{\hss$\m@th#1\mid$\hbox to 0pt{$\m@th#1^{#2}$\hss}\hss}%
  \lower.9\ht0\hbox to 0pt{\hss$\m@th#1\smile$\hss}%
  \kern\wd0}

\def\dotminussym#1#2{%
  \setbox0=\hbox{$\m@th#1-$}%
  \kern.5\wd0%
  \hbox to 0pt{\hss\hbox{$\m@th#1-$}\hss}%
  \raise.6\ht0\hbox to 0pt{\hss$\m@th#1.$\hss}%
  \kern.5\wd0}

\renewcommand{\emptyset}{\varnothing}
\renewcommand{\setminus}{\smallsetminus}

\usepackage{mathpazo}

\newcommand{\rest}{{\restriction}}

\newcommand{\half}[1][1]{\tfrac{#1}{2}}

\DeclareMathOperator{\tp}{tp}

\newcommand{\qf}{\mathrm{qf}}

\DeclareMathOperator{\Th}{Th}

\DeclareMathOperator{\tS}{S}
\newcommand{\tSqf}{\tS^\qf}

\DeclareMathOperator{\diam}{diam}
\DeclareMathOperator{\co}{co}
\newcommand{\cco}{\overline\co}

\newcommand{\id}{\mathrm{id}}

\DeclareMathOperator{\Emb}{Emb}

\DeclareMathOperator{\Aut}{Aut}

\DeclareMathOperator{\dom}{dom}

\newcommand{\cE}{\mathcal{E}}

\newcommand{\sI}{\mathscr{I}}

\newcommand{\bG}{\mathbf{G}}

\newcommand{\bM}{\mathbf{M}}
\newcommand{\bN}{\mathbf{N}}

\newcommand{\bR}{\mathbf{R}}

\makeatother

\usepackage[arrow,curve]{xy}
\everyxy={0;<1cm,0cm>:<0cm,-1cm>::}

\DeclareMathOperator{\epi}{epi}

\begin{document}

\title{Generic orbits and type isolation in the Gurarij space}

\author{Itaï \textsc{Ben Yaacov}}

\address{Itaï \textsc{Ben Yaacov} \\
  Université Claude Bernard -- Lyon 1 \\
  Institut Camille Jordan, CNRS UMR 5208 \\
  43 boulevard du 11 novembre 1918 \\
  69622 Villeurbanne Cedex \\
  France}

\urladdr{\url{http://math.univ-lyon1.fr/~begnac/}}

\thanks{Research supported by the Institut Universitaire de France, by ANR project GruPoLoCo (ANR-11-JS01-008) and by a grant from the Simons Foundation (\#202251 to the second author).}
\thanks{Part of the work was carried out during the Universality and Homogeneity programme at the Hausdorff Institute of the University of Bonn (Fall 2013).}

\author{C. Ward \textsc{Henson}}

\address{C. Ward \textsc{Henson} \\
  University of Illinois at Urbana-Champaign \\
  Urbana, Illinois 61801 \\
  USA}

\urladdr{\url{http://www.math.uiuc.edu/~henson/}}

\svnInfo $Id: Gurarij.tex 2822 2016-03-11 14:20:50Z begnac $
\thanks{\textit{Revision} {\svnInfoRevision} \textit{of} \today}

\keywords{Gurarij space ; Banach space ; isolated type ; atomic model ; prime model ; group action ; generic orbit}
\subjclass[2010]{46B04 ; 03C30 ; 03C50 ; 03C98}

\begin{abstract}
  We study the question of when the space of embeddings of a separable Banach space $E$ into the separable Gurarij space $\bG$ admits a generic orbit under the action of the linear isometry group of $\bG$.
  The question is recast in model-theoretic terms, namely type isolation and the existence of prime models.
  We characterise isolated types over $E$ using tools from convex analysis.
  We show that if the set of isolated types over $E$ is dense, then a dense $G_\delta$ orbit exists, and otherwise all orbits are meagre.
  We then study some (families of) examples with respect to this dichotomy.
  We also point out that the class of Gurarij spaces is the class of models of an $\aleph_0$-categorical theory with quantifier elimination, and calculate the density character of the space of types over $E$, answering a question of Avilés et al.
\end{abstract}

\maketitle

\tableofcontents

\section*{Introduction}

In 1966, Gurarij \cite{Gurarij:UniversalPlacement} defined what came to be known as the (separable) Gurarij space, and proved that it is almost isometrically unique.
The isometric uniqueness of the Gurarij space was proved in 1976 by Lusky \cite{Lusky:UniqueGurarij}.
In the same paper, Lusky points out that his arguments could be modified to prove also the isometric uniqueness of the separable Gurarij space equipped with a distinguished smooth unit vector (namely, a unit that admits a unique norming linear functional, see \autoref{dfn:NormingLinearFunctional}).
In other words, if $\bG$ denotes the separable Gurarij space, then the set of smooth unit vectors in $\bG$ forms an orbit under the action of the linear isometry group $\Aut(\bG)$.
By Mazur \cite{Mazur:Konvexe}, this orbit is moreover a dense $G_\delta$ subset of the unit sphere.

These facts are strongly reminiscent of familiar model theoretic phenomena, and, as we show in this paper, are indeed special cases thereof.
It was observed some time ago by the second author that the uniqueness of the Gurarij space can be accounted for by it being the unique separable model of an $\aleph_0$-categorical theory, which moreover eliminates quantifiers.
Similarly, the Gurarij space is atomic over a vector if and only if the latter is smooth, so Lusky's second uniqueness result is a special case of the uniqueness of the prime model.

These observations serve as a starting point for the present paper, whose goals are threefold:
\begin{itemize}
\item Make the observations above precise, and generalise them to uniqueness results over a subset other than the empty set or a singleton -- in other words, we study uniqueness and primeness of the Gurarij space over a subspace $E$ of dimension possibly greater than one.
  As we shall see, this requires us to (define and) characterise when types over $E$ are isolated.
\item Present in a manner accessible to non-logicians, and without making use of formal logic, some tools and techniques of model theory: types, type spaces, type isolation, the Tarski-Vaught Criterion, the Omitting Types Theorem, atomic models, and the primeness and uniqueness of atomic models.
\item Present to model theorists, who are familiar with the tools mentioned in the previous item in the context of classical logic, how these tools adapt to the metric setting.
\end{itemize}

\medskip

In \autoref{sec:QuantifierFreeTypes} we recall the definition of (quantifier-free) types and type spaces over a Banach space $E$, and study their properties.
The topometric structure of the type space, a fundamental notion of metric model theory, is defined there, as well as (topometrically) isolated types, which are among the main objects of study of this paper.

In \autoref{sec:Gurarij} we start studying Gurarij spaces.
At the technical level, we define and study Gurarij (and other) spaces that are atomic over a fixed separable parameter space $E$, and prove the Omitting Types Theorem (\autoref{thm:GurarijOmittingTypes}).
We prove appropriate generalisations of the homogeneity and universality properties of the Gurarij space to homogeneity and universality over $E$.
In particular, we show that the prime Gurarij spaces over $E$ (see \autoref{cor:PrimenessCriterion}) are those Gurarij spaces that are separable and atomic over $E$.
If a prime Gurarij space over $E$ exists then it is unique, up to an isometric isomorphism over $E$, and is denoted $\bG[E]$.
When $\bG[E]$ exists, the set of embeddings of $E$ in $\bG$ over which $\bG$ is prime forms a dense $G_\delta$ orbit among all embeddings of $E$; otherwise, all orbits are meagre.
We also give the standard model theoretic criterion for the existence of $\bG[E]$, namely that the isolated types over $E$ are dense.

While this (re-)development of model-theoretic tools is carried out in a fairly specific context, we present arguments that would be valid in the general case; these are sometimes followed by separate results that improve the general ones in the specific context of the Gurarij space.
The few results that do make explicit use of formal logic (essentially, \autoref{prp:BanachModelCompletion} and \autoref{thm:GurarijTheory}) serve mostly as parenthetical remarks required for the sake of completeness, and are not used in any way in other parts of the paper.

At this point we move on to the questions of when $\bG[E]$ exists and how to characterise isolated types in a fashion suitable to the Banach space context.
In \autoref{sec:IsolatedTypesDimensionOne} we consider the special case where $\dim E = 1$, giving a model-theoretic account of Lusky's result about smooth points in $\bG$.
Before considering the general case, we introduce an essential tool in \autoref{sec:LegendreFenchel}, namely the presentation of $1$-types as convex Katětov functions (as per \cite{BenYaacov:UniversalGurarijIsometryGroup}), and the Legendre-Fenchel transformation of those.
In \autoref{sec:IsolatedTypes} we characterise isolated $1$-types in terms of their Legendre-Fenchel conjugate, which allows us to give in \autoref{sec:Existence} some sufficient conditions for the existence of $\bG[E]$ for finite-dimensional spaces $E$ (e.g., smooth, polyhedral, or of dimension $\leq 3$ -- see \autoref{thm:IsolatedTypesFiniteDimension}), as well as examples when $\bG[E]$ does not exist.
The question of a satisfactory necessary and sufficient condition on $E$ for the existence of $\bG[E]$ remains an open problem.

We conclude in \autoref{sec:CountingTypes} with a ``counting types'' result, showing that the space of types over $E$ is metrically separable if and only if $E$ is finite-dimensional and polyhedral.
This allows us to answer a question of Avilés et al.\ \cite{Aviles-CabelloSanchez-Castillo-Gonzalez-Moreno:UniversalDisposition}.

\medskip

Throughout, $E$, $F$ and so on denote vector spaces over the real numbers -- normed spaces, most of the time, although the Legendre-Fenchel duality in \autoref{sec:LegendreFenchel} is stated for general locally convex spaces.
An embedding (or isomorphism, automorphism) of normed spaces is always isometric.

The topological dual of a normed space $E$ will be denoted $E^*$.
We shall often use the notation $B(E)$ for the closed unit ball of $E$ and $\partial B(E)$ for the unit sphere (which, regardless of topology, is the boundary of $B(E)$ in the sense of convex geometry), and similarly for $E^*$ instead of $E$.

\section{Quantifier-free types in Banach spaces}
\label{sec:QuantifierFreeTypes}

Before we start, let us state the following basic amalgamation result, which we shall use many times, quite often implicitly.

\begin{fct}
  \label{fct:BanachSpaceAmalgamation}
  For any three Banach spaces $E_0$, $F_0$ and $F_1$, and isometric embeddings $f_i\colon E_0 \rightarrow F_i$, there is a fourth Banach space $E_1$ and isometric embeddings $g_i\colon F_i \rightarrow E_1$ such that $g_0 f_0 = g_1 f_1$.
\end{fct}
\begin{proof}
  Equip the direct sum $F_0 \oplus F_1$ with the semi-norm $\|v+u\| = \inf_{w \in E_0} \|v + f_0 w\| + \|u - f_1 w\|$, divide by the kernel and complete.
\end{proof}

We can now define the fundamental objects of study of this section and, to a large extent, the entire paper.

\begin{dfn}
  \label{dfn:Type}
  Let $E$ be a Banach space and $X$ a sequence of distinct symbols (indexed by an arbitrary set $I$) that we call \emph{variables}.
  We let $E(X) = E \oplus \bigoplus_{x \in X} \bR x$, and define $\tS_X(E)$ to consist of all semi-norms on $E(X)$ that extend the norm on $E$, calling it the \emph{space of types in $X$ over $E$}.
  We denote members of $\tS_X(E)$ by $\xi$, $\zeta$ and so on, and the corresponding semi-norms by $\|{\cdot}\|^\xi$, $\|{\cdot}\|^\zeta$ and so on.

  When $X = \{x_i\}_{i \in I}$ we may also write $E(I) = E \oplus \bigoplus_{i \in I} \bR x_i$ instead of $E(X)$, and similarly $\tS_I(E)$, whose members are called $I$-types.
\end{dfn}

\begin{dfn}
  \label{dfn:TypeOfTuple}
  Given a Banach space extension $E \subseteq F$ and an $I$-sequence $\bar a = \{a_i\}_{i \in I} \subseteq F$, we define the \emph{type of $\bar a$ over $E$}, in symbols $\xi = \tp(\bar a/E) \in \tS_I(E)$, to be the semi-norm $\left\| b + \sum \lambda_i x_i \right\|^\xi = \left\| b + \sum \lambda_i a_i \right\|$, and say that $\bar a$ \emph{realises} $\xi$.
  When a sequence $\bar b$ generates $E$, we may also write $\tp(\bar a/\bar b)$ for $\tp(\bar a/E)$.

  Conversely, given a type $\xi \in \tS_I(E)$, we define the Banach space \emph{generated} by $\xi$, in symbols $E[\xi]$, as the space obtained from $\bigl( E(I), \|{\cdot}\|^\xi \bigr)$ by dividing by the kernel and completing, together with the distinguished generators $\{x_i\}_{i \in I} \subseteq E[\xi]$.
\end{dfn}

\begin{rmk}
  A model-theorist will recognise types as we define them here as \emph{quantifier-free} types, which do not, in general, capture ``all the pertinent information''.
  However, by \autoref{fct:BanachSpaceAmalgamation}, they do capture a maximal existential type.
  Moreover, it follows from \autoref{lem:VariableChange} below (and more specifically, from the assertion that $\pi_{\bar x}\colon \tS_{\bar x,y}(0) \rightarrow \tS_{\bar x}(0)$ is open) that being an existentially closed Banach space is an elementary property, so the theory of Banach spaces admits a model companion.
  Then \autoref{fct:BanachSpaceAmalgamation} can be understood to say that the model companion eliminates quantifiers, so quantifier-free types and types are in practice the same -- see \autoref{prp:BanachModelCompletion}.
  As we shall see later, the model companion is separably categorical, and its unique separable model is $\bG$, the separable Gurarij space.
\end{rmk}

\begin{dfn}
  \label{dfn:TypeSpaceTopometricStructure}
  We equip $\tS_I(E)$ with a topological structure as well as with a metric structure, \emph{which may be distinct}.
  The \emph{topology} on $\tS_I(E)$ is the least one in which, for every member $x \in E(I)$, the map $\hat x\colon \xi \mapsto \|x\|^\xi$ is continuous.
  Given $\xi,\zeta \in \tS_I(E)$, we define the \emph{distance} $d(\xi,\zeta)$ to be the infimum, over all $F$ extending $E$ and over all realisations $\bar a$ and $\bar b$ of $\xi$ and $\zeta$, respectively, of $\sup_i \|a_i - b_i\|$.
\end{dfn}

It is fairly clear that:
\begin{enumerate}
\item The distance on $\tS_I(E)$ refines the topology.
\item While the distance need not agree with the topology (we shall see that unless the parameter space $E$ is trivial and $I$ is finite, they are in fact distinct), it is lower semi-continuous.
\end{enumerate}
In other words, $\tS_I(E)$, equipped with this double structure, is a \emph{topometric space} in the sense of \cite{BenYaacov:TopometricSpacesAndPerturbations}.

\begin{lem}
  \label{lem:ApproximateTypeAmalgamation}
  Let $E,F$ be Banach spaces, $I$ an index set, and consider tuples $\bar a = (a_i)_{i\in I} \in E^I$, $\bar b \in F^I$ and $\bar \varepsilon \in \bR^I$.
  Also let $\bR^{(I)}$ denote the set of all $I$-tuples in which all but finitely many entries vanish.
  The following conditions are equivalent.
  \begin{enumerate}
  \item
    There exists a semi-norm $\|{\cdot}\|$ on $E \oplus F$ extending the respective norms of $E$ and $F$, such that for each $i \in I$ one has $\|a_i - b_i\| \leq \varepsilon_i$.
  \item
    For all $\bar r \in \bR^{(I)}$, one has
    \begin{gather*}
      \Biggl| \left\| \sum r_i a_i \right\| - \left\| \sum r_i b_i \right\| \Biggr| \leq \sum |r_i| \varepsilon_i.
    \end{gather*}
  \end{enumerate}
\end{lem}
\begin{proof}
  One direction being trivial, we prove the other.
  For $c + d \in E \oplus F$ define
  \begin{gather*}
    \|c + d\|' = \inf_{\bar r \in \bR^{(I)}} \left\| c - \sum r_i a_i \right\| + \left\| d + \sum r_i b_i \right\| + \sum |r_i|\varepsilon_i.
  \end{gather*}
  This is easily checked to be a semi-norm, with $\|c\|' \leq \|c\|$ for $c \in E$.
  Now, for $c \in E$ and $\bar r \in \bR^{(I)}$ we have
  \begin{gather*}
    \left\| c - \sum r_i a_i \right\| + \left\| \sum r_i b_i \right\| + \sum |r_i|\varepsilon_i \geq \left\| c - \sum r_i a_i \right\| + \left\| \sum r_i a_i \right\| \geq \|c\|.
  \end{gather*}
  Therefore $\|c\|' = \|c\|$, and similarly $\|d\|' = \|d\|$ for $d \in F$, concluding the proof.
\end{proof}

\begin{prp}
  \label{prp:TypeDistance}
  Let $\xi,\zeta \in \tS_I(E)$, and let $E(I)_1$ consist of all $a + \sum \lambda_i x_i \in E(I)$ (where $a \in E$ and all but finitely many of the $\lambda_i$ vanish) such that $\sum |\lambda_i| = 1$.
  Then
  \begin{gather*}
    d(\xi,\zeta) = \sup_{x \in E(I)_1} \left| \|x\|^\xi - \|x\|^\zeta \right|.
  \end{gather*}
  Moreover, the infimum in the definition of distance between types is attained.
\end{prp}
\begin{proof}
  Immediate from \autoref{lem:ApproximateTypeAmalgamation}.
\end{proof}

\begin{conv}
  \label{conv:Topometric}
  When referring to the topological or metric structure of $\tS_I(E)$, we shall follow the convention that unqualified terms from the vocabulary of general topology (open, compact, and so on) apply to the topological structure, while terms specific to metric spaces (bounded, complete, and so on) refer to the metric structure.

  Excluded from this convention is the notion of isolation which will be defined in a manner that takes into account both the topology and the distance.
\end{conv}

While this convention may seem confusing at first, it is quite convenient, as in the following.

\begin{lem}
  \label{lem:TypeSpaceTopologyMetric}
  \begin{enumerate}
  \item
    The space $\tS_I(E)$ is Hausdorff, and every closed and bounded set thereof is compact.
  \item
    The distance on $\tS_I(E)$ is lower semi-continuous.
    In particular, the closure of a bounded set is bounded.
  \item
    Assume that $I$ is finite, say $I = n = \{0,1,\ldots,n-1\} \in \bN$.
    Then every bounded set in $\tS_n(E)$ is contained in an open bounded set.
    It follows that the space $\tS_n(E)$ is locally compact, and that a compact subset is necessarily (closed and) bounded.
  \item
    A subset $X \subseteq \tS_n(E)$ is closed if and only if its intersection with every compact set is compact.
  \end{enumerate}
\end{lem}
\begin{proof}
  For the first item, clearly $\tS_I(E)$ is Hausdorff.
  If $X \subseteq \tS_I(E)$ is bounded, then for every $x \in E(I)$ there exists $M_x$ such that $\|x\|^\xi \leq M_x$ for all $\xi \in X$.
  We can therefore identify $X$ with a subset of $Y = \prod_x [0,M_x]$, and if $X$ is closed in $\tS_I(E)$, then it is closed in $Y$ and therefore compact.

  The second item follows from \autoref{prp:TypeDistance}, and the third is immediate.

  For the fourth item, assume that $X \subseteq \tS_n(E)$ is not closed, let $\xi \in \overline X \setminus X$ and let $U$ be a bounded neighbourhood of $\xi$, in which case $\overline U \cap X$ is not compact.
\end{proof}

When $I$ is infinite, the distance on $S_I(E)$ is somewhat badly behaved: it can be infinite, and parts of \autoref{lem:TypeSpaceTopologyMetric} may fail.
Using it will become even more problematic for finer notions considered below, such as type isolation.
Henceforth we shall only consider the distance between types when $I$ is finite.

\begin{dfn}
  \label{lem:VariableRestriction}
  Let $m \leq n$.
  The \emph{variable restriction map} $\pi\colon \tS_n(E) \rightarrow \tS_m(E)$ is the natural one induced by the inclusion $E(x_0,\ldots,x_{m-1}) \subseteq E(x_0,\ldots,x_{n-1})$, namely $\|y\|^{\pi \xi} = \|y\|^\xi$ for $y \in E(x_0,\ldots,x_{m-1})$.
\end{dfn}

\begin{lem}
  \label{lem:VariableRestriction}
  Let $m \leq n$, and let $\pi\colon \tS_n(E) \rightarrow \tS_m(E)$ denote the variable restriction map.
  Then for every $\xi \in \tS_n(E)$ and $\zeta \in \tS_m(E)$ we have $d(\pi \xi,\zeta) = d(\xi,\pi^{-1}\zeta)$.
  Moreover there exists $\rho \in \pi^{-1}\zeta$ such that $d(\pi \xi,\zeta) = d(\xi,\rho)$ and $\|x_i\|^\rho = \|x_i\|^\xi$ for all $m \leq i < n$.

  In particular, the map $\pi$ is metrically open.
\end{lem}
\begin{proof}
  The inequality $d(\pi \xi,\zeta) \leq d(\xi,\pi^{-1}\zeta)$ is immediate.
  For the opposite inequality, assume that $d(\pi \xi, \zeta) < r$.
  By definition, there exist an extension $F \supseteq E$ and realisations $\bar a$ of $\pi \xi$ and $\bar b$ of $\zeta$ in $F$ such that $\|a_i - b_i\| < r$ for $i < m$.
  By \autoref{fct:BanachSpaceAmalgamation}, possibly extending $F$, there is $\bar c \in F^{n-m}$ such that $\tp(\bar a \bar c) = \xi$.
  Then $\rho = \tp(\bar b\bar c/E)$ is as desired for both the main assertion and the moreover part.
  It follows that $\pi B(\xi,r) \supseteq B(\pi \xi,r)$, so $\pi$ is metrically open.
\end{proof}

\begin{dfn}
  \label{dfn:Isolation}
  We say that a type $\xi \in \tS_n(E)$ is \emph{isolated} if the distance and the topology agree at $\xi$, i.e., if every metric neighbourhood of $\xi$ is also a topological one.
\end{dfn}

This is the definition of isolation in a topometric space, taking into account both the metric structure and the topological structure.
Ordinary topological spaces can be viewed as topometric spaces by equipping them with the discrete $0/1$ distance, in which case the notion of isolation as defined here coincides with the usual one.

Many results regarding ordinary topological spaces still hold, when translated correctly, with the topometric definitions.
For example, the fact that a dense set must contain all isolated points becomes the following.
Notice that in \autoref{lem:IsolatedTypesMetricallyClosed} below we prove that the set of isolated types is itself metrically closed.

\begin{lem}
  \label{lem:DenseIsolatedTypes}
  Let $E$ be a Banach space, $D \subseteq \tS_n(E)$ a dense, metrically closed set.
  Then $D$ contains all isolated types.
\end{lem}
\begin{proof}
  If $\xi$ is isolated, then all metric neighbourhoods of $\xi$, which are also topological neighbourhoods, must intersect $D$.
\end{proof}

For reasons that will become clearer in \autoref{sec:Gurarij}, one of our main goals in this paper is to characterise isolated types.
We start with the easiest situation.

\begin{prp}
  \label{prp:SeparableCategoricity}
  Let $0$ denote the trivial Banach space.
  Then every type in $\tS_n(0)$ is isolated.
  In other words, the distance on $\tS_n(0)$ agrees with the topology.
\end{prp}
\begin{proof}
  Given $N \in \bN$, let $X_N \subseteq 0(n)_1$ be the finite set consisting of all $\sum \lambda_i x_i$ where $\sum |\lambda_i| = 1$ and each $\lambda_i$ is of the form $\frac{k}{N}$.
  For $\xi \in \tS_n(0)$, let $U_{\xi,N}$ be its neighbourhood consisting of all $\zeta$ such that
  \begin{gather*}
    \forall x \in X_N \qquad \|x\|^\xi - 1/N < \|x\|^\zeta < \|x\|^\xi + 1/N.
  \end{gather*}
  This means in particular that $\|x_i\|^\zeta < \|x_i\|^\xi + 1$ for all $i < n$, and now an easy calculation together with \autoref{prp:TypeDistance} yields that there exists a constant $C(\xi)$ such that for all $N$, $U_{\xi,N}$ is contained in the ball of radius $C(\xi)/N$ around $\xi$, which is what we had to show.
\end{proof}

This already allows us to construct the following useful tool of variable change in a type.

\begin{dfn}
  \label{dfn:VariableChange}
  \begin{enumerate}
  \item Given a linear map $\varphi\colon E(\bar y) \rightarrow E(\bar x)$ extending $\id_E$, we define a pull-back map $\varphi^*\colon \tS_{\bar x}(E) \rightarrow \tS_{\bar y}(E)$, or $\varphi^*\colon \tS_n(E) \rightarrow \tS_m(E)$, by
    \begin{gather*}
      \|w\|^{\varphi^* \xi} = \| \varphi w \|^\xi, \qquad w \in E(\bar y).
    \end{gather*}
    For $A \subseteq \tS_{\bar y}(E)$ we define $\varphi_* A =(\varphi^*)^{-1}(A) \subseteq \tS_{\bar x}(E)$.
  \item Given a tuple $\bar z$ in $E(\bar x)$, of the same length as $\bar y$, define $\varphi\colon E(\bar y) \rightarrow E(\bar x)$ to be the unique linear map extending $\id_E$ and sending $y_i \mapsto z_i$.
    We then write $\xi\rest_{\bar z} = \varphi^* \xi$, so
    \begin{gather*}
      \left\| a + \sum \lambda_i y_i \right\|^{\xi\rest_{\bar z}} = \left\| a + \sum \lambda_i z_i \right\|^\xi.
    \end{gather*}
    (In this notation $E$, $\bar x$ and $\bar y$ are assumed to be known from context.)
  \end{enumerate}
\end{dfn}

\begin{lem}
  \label{lem:VariableChange}
  For a fixed tuple $\bar y \in E(\bar x)^m$, possibly with repetitions, the restriction map $\tS_n(E) \rightarrow \tS_m(E)$, $\xi \mapsto \xi\rest_{\bar y}$, is continuous and Lipschitz (here $n = |\bar x|$).
  If $\bar y$ are linearly independent over $E$, then this map is also topologically and metrically open.
  Moreover, the metric openness is ``Lipschitz'' as well, in the sense that there exists a constant $C = C(\bar y)$ such that for all $\xi$ and all $r > 0$ we have
  \begin{gather*}
    B(\xi,r) \rest_{\bar y} \supseteq B(\xi \rest_{\bar y},Cr).
  \end{gather*}
\end{lem}
\begin{proof}
  Continuity and the Lipschitz condition are easy.
  We therefore assume that $\bar y$ are linearly independent over $E$, in which case we may also view them as formal unknowns.
  This gives rise to an inclusion map $E(\bar y) \hookrightarrow E(\bar x)$, and $\xi \mapsto \xi\rest_{\bar y}$ may be viewed as a map $\tS_{\bar x}(E) \rightarrow \tS_{\bar y}(E)$.
  In the special case where $\bar y$ generate $E(\bar x)$ over $E$ we have $E(\bar x) = E(\bar y)$, and the Lipschitz map $\tS_{\bar y}(E) \rightarrow \tS_{\bar x}(E)$, $\xi \mapsto \xi\rest_{\bar x}$ is the inverse of $\xi \mapsto \xi\rest_{\bar y}$ giving the moreover part.
  In the special case where $y_i = x_i$ for $i < m$, the moreover part follows from \autoref{lem:VariableRestriction}.
  In the general case, we may complete $\bar y$ to a basis for $E(\bar x)$ over $E$, and the moreover part follows as a composition of the two special cases.

  For topological openness, we proceed as follows.
  In the case where $E = 0$, this follows from metric openness and \autoref{prp:SeparableCategoricity}.
  Let us consider now the case where $E$ is finite-dimensional.
  We fix a basis $\bar b$ for $E$ and a corresponding tuple of variables $\bar w$.
  We may then identify $E(\bar x)$ with $0(\bar w,\bar x)$, and thus $\bar y$ with its image in $0(\bar w,\bar x)$.
  We already know that $\cdot\rest_{\bar w,\bar y}\colon \tS_{\bar w,\bar x}(0) \rightarrow \tS_{\bar w,\bar y}(0)$ is open.
  In addition, we have a commutative diagram
  \begin{gather*}
    \begin{xy}
      (0,0)*+{\tS_{\bar w,\bar x}(0)}="Swx",
      (4,0)*+{\tS_{\bar w,\bar y}(0)}="Swy",
      (2,2)*+{\tS_{\bar w}(0)}="Sw",
      \ar^{\cdot\rest_{\bar w,\bar y}} "Swx";"Swy",
      \ar_{\cdot\rest_{\bar w}} "Swx";"Sw",
      \ar^{\cdot\rest_{\bar w}} "Swy";"Sw"
    \end{xy}
  \end{gather*}
  and the map $\cdot\rest_{\bar y}\colon \tS_{\bar x}(E) \rightarrow \tS_{\bar y}(E)$ is homeomorphic to the fibre of the horizontal arrow over $\tp(\bar b) \in \tS_{\bar w}(0)$, so it is open as well.
  The infinite-dimensional case follows from the finite-dimensional one, since any basic open set in $\tS_{\bar x}(E)$ can be defined using finitely many parameters in $E$.
\end{proof}

We leave it to the reader to check that if $\bar y$ are not linearly independent over $E$, then $\cdot\rest_{\bar y}$ is not metrically open, and \textit{a fortiori} not topologically so (consider, for example $\cdot\rest_{x,x}\colon \tS_1(0) \rightarrow \tS_2(0)$).

\begin{lem}
  \label{lem:TopometricOpen}
  Let $U \subseteq \tS_n(E)$ be open and $r > 0$.
  Then $B(U,r) = \{\xi : d(\xi,U) < r\} \subseteq \tS_n(E)$ is open as well.
\end{lem}
\begin{proof}
  Let $\bar x$ and $\bar y$ be two $n$-tuples of variables.
  Let us identify $\tS_n(E)$ with $\tS_{\bar x}(E)$, and let $W \subseteq \tS_{\bar x,\bar y}(E)$ consist of all $\xi$ such that $\| x_i - y_i \|^\xi < r$ for $i < n$.
  Then $W$ is open, and by \autoref{lem:VariableChange} so is $V = \bigl( W \cap (\cdot\rest_{\bar x})^{-1}(U) \bigr)\rest_{\bar y} \subseteq \tS_{\bar y}(E)$.
  Identifying $\tS_{\bar y}(E)$ with $\tS_n(E)$ as well, $V = B(U,r)$.
\end{proof}

\begin{lem}
  \label{lem:IsolatedTypesMetricallyClosed}
  Let $E$ be a Banach space.
  \begin{enumerate}
  \item A type in $\tS_n(E)$ is isolated if and only if all its metric neighbourhoods have non empty interior.
  \item The set of isolated types in $\tS_n(E)$ is metrically closed.
  \end{enumerate}
\end{lem}
\begin{proof}
  The first assertion follows easily from \autoref{lem:TopometricOpen}, and the second from the first.
\end{proof}

Another basic operation one can consider on types is the \emph{restriction of parameters} $\tS_n(F) \rightarrow \tS_n(E)$ when $E \subseteq F$.

\begin{lem}
  \label{lem:ParameterRestriction}
  Let $E \subseteq F$ be an isometric inclusion of Banach spaces.
  Then the natural type restriction map $\theta\colon \tS_n(F) \rightarrow \tS_n(E)$ is continuous, closed, and satisfies $\theta B(\xi,r) = B(\theta \xi,r)$.

  In particular, $\theta$ is both topologically and metrically a quotient map.
\end{lem}
\begin{proof}
  It is clear that $\theta$ is continuous.
  To see that it is closed we use \autoref{lem:TypeSpaceTopologyMetric}.
  Indeed, since closed sets are exactly those that intersect compact sets on compact sets, it will be enough to show that if $K \subseteq \tS_n(E)$ is compact, then so is $\theta^{-1} K$, which follows from the characterisation of compact sets as closed and bounded.

  It is clear that $\theta B(\xi,r) \subseteq B(\theta \xi,r)$.
  Conversely, if $\zeta_0 \in \tS_n(E)$, then using \autoref{fct:BanachSpaceAmalgamation}, there exists $\zeta \in \theta^{-1}\zeta_0$ with $d(\xi,\zeta) \leq d(\theta \xi,\zeta_0)$, which proves that $\theta B(\xi,r) = B(\theta \xi,r)$.
\end{proof}

We also obtain that the theory $T$ of (unit balls of) Banach spaces admits a model completion $T^*$, namely a companion with quantifier elimination, whose types are exactly those defined above.
Moreover, as we shall see in the following section, the models of $T^*$ are exactly the Gurarij spaces.
Since this is somewhat of an aside with respect to the rest of this paper, we shall allow ourselves to be brief, and assume that the reader is familiar with continuous first order logic (see \cite{BenYaacov-Usvyatsov:CFO,BenYaacov-Berenstein-Henson-Usvyatsov:NewtonMS}), and, for the part regarding Banach spaces as unbounded metric structures, also with unbounded continuous logic (see \cite{BenYaacov:Unbounded}).

\begin{lem}
  \label{lem:QuantifierEliminatingTheoryFromTypeSpace}
  Let $T$ be an inductive theory in continuous first order logic, and for $n \in \bN$ let $\tSqf_n(T)$ denote the space of quantifier-free types consistent with $T$, equipped with the natural logic topology.
  Assume first, that every two models of $T$ amalgamate over a common substructure, and second, that for every $n$, the variable restriction map $\tSqf_{n+1}(T) \rightarrow \tSqf_n(T)$ is open.
  Then $T$ admits a model completion, namely a companion that eliminates quantifiers.

  (In fact, an approximate amalgamation property for models of $T$ over a common finitely generated substructure suffices.)
\end{lem}
\begin{proof}
  Let $\varphi(\bar x,y)$ be a quantifier-free formula, inducing a continuous function $\hat \varphi\colon \tSqf_{n+1}(T) \rightarrow \bR$ (which has compact range, by compactness of $\tSqf_{n+1}(T)$).
  Let $\pi\colon \tSqf_{n+1}(T) \rightarrow \tSqf_n(T)$ denote the variable restriction map, and define $\rho\colon \tSqf_n(T) \rightarrow \bR$ as the infimum over the fibre:
  \begin{gather*}
    \rho(q) = \inf \bigl\{ \hat \varphi(p) : \pi p = q \bigr\}.
  \end{gather*}
  Since $\pi$ is continuous (automatically) and open (by hypothesis), $\rho$ is continuous as well, and can therefore be expressed as a uniform limit of $\hat \psi_n\colon \tSqf_n(T) \rightarrow \bR$, where $\psi_n(\bar x)$ are quantifier-free formulae, say $\|\rho - \hat \psi_n\| \leq 2^{-n}$ for all $n$.
  One can now express that $\sup_{\bar x} \, |\psi_n(\bar x) - \inf_y \, \varphi(\bar x,y)| \leq 2^{-n}$ for all $n$ by a set of sentences.

  Let $T^*$ consist of $T$ together with all sentences constructed as above, for all possible quantifier-free formulae $\varphi(\bar x,y)$.
  Then every existentially closed model of $T$ is easily checked to be a model of $T^*$ (using our amalgamation hypothesis), so $T$ and $T^*$ are companions.
  Moreover, by induction on quantifiers, every formula is equivalent modulo $T^*$ to a uniform limit of quantifier-free formulae, so $T^*$ eliminates quantifiers.
\end{proof}

\begin{prp}
  \label{prp:BanachModelCompletion}
  Consider Banach spaces either as metric structures in unbounded continuous logic, or as bounded metric structures via their closed unit balls, as explained, say, in \cite{BenYaacov:NakanoSpaces}.
  Then (in either approach) the theory of the class of Banach spaces is inductive, and admits a model completion $T^*$ which is moreover complete and $\aleph_0$-categorical.

  When the entire Banach space is viewed as a structure, then the types over a subspace are as per \autoref{dfn:Type} and \autoref{dfn:TypeOfTuple}, and if one only considers the unit ball, then the space of $I$-types over $B(E)$ is $\tS_I^{\leq 1}(E) = \bigl\{ \xi \in \tS_I(E)\colon \|x_i\|^\xi \leq 1 \text{ for all } i \in I \bigr\}$.
\end{prp}
\begin{proof}
  Let us consider the theory $T$ of unit balls of Banach spaces.
  It is clearly inductive, and it is fairly easy to check that the space of quantifier-free $I$-types over a unit ball $B(E)$ is the space $\tS^{\leq 1}_I(E)$ defined in the statement.
  By the moreover part of \autoref{lem:VariableRestriction}, variable restriction $\tS^{\leq 1}_{n+1}(E) \rightarrow \tS^{\leq 1}_n(E)$ is metrically open.
  For $E = 0$ this implies in particular that $\tS^{\leq 1}_{n+1}(0) \rightarrow \tS^{\leq 1}_n(0)$ is topologically open, but this latter is just $\tSqf_{n+1}(T) \rightarrow \tSqf_n(T)$.
  This, together with \autoref{fct:BanachSpaceAmalgamation}, fulfils the hypotheses of \autoref{lem:QuantifierEliminatingTheoryFromTypeSpace}.
  By quantifier elimination, $\tS_n(T^*) = \tSqf_n(T) = \tS^{\leq 1}_n(0)$, so in particular, $\tS_0(T^*)$ is a singleton, whereby $T^*$ is complete.
  Finally, $T^*$ is $\aleph_0$-categorical by the Ryll-Nardzewski Theorem and the fact that all types over the trivial space are isolated (see \cite{BenYaacov-Usvyatsov:dFiniteness}).

  The case of Banach spaces as unbounded structures follows via the bi-interpretability of the whole Banach space with its unit ball.
\end{proof}

\section{The Gurarij space}
\label{sec:Gurarij}

\begin{dfn}
  \label{dfn:Gurarij}
  We recall from, say, Lusky \cite{Lusky:UniqueGurarij} that a \emph{Gurarij space} is a Banach space $\bG$ having the property that for any $\varepsilon > 0$, finite-dimensional Banach space $E \subseteq F$, and isometric embedding $\varphi\colon E \rightarrow \bG$, there is a linear map $\psi\colon F \rightarrow \bG$ extending $\varphi$ such that in addition, for all $x \in F$, $(1 - \varepsilon)\|x\| \leq \|\psi x\| \leq (1 + \varepsilon)\|x\|$.
\end{dfn}

Some authors add the requirement that a Gurarij space be separable, but from our point of view it seems more elegant to consider separability as a separate property.

\begin{lem}
  \label{lem:GurarijTarskiVaught}
  Let $F$ be a Banach space.
  Then the following are equivalent:
  \begin{enumerate}
  \item The space $F$ is a Gurarij space.
  \item \label{item:GurarijTarskiVaughtN}
    For every $n$, the set of realised types $\tp(\bar a/F)$, as $\bar a$ varies over $F^n$, is dense in $\tS_n(F)$.
  \item Same as \autoref{item:GurarijTarskiVaughtN} for $n = 1$.
  \end{enumerate}
\end{lem}
\begin{proof}
  \begin{cycprf}
  \item[\impnum{1}{3}]
    Let $U \subseteq \tS_1(F)$ be open and $\xi \in U$.
    We may assume that $U$ is defined by a finite set of conditions of the form $\bigl| \| a_i + r_i x \| - 1 \bigr| < \varepsilon$, where $\| a_i + r_i x \|^\xi = 1$.
    Let $E \subseteq F$ be the subspace generated by the $a_i$, and let $E' = E + \bR x$ be the extension of $E$ generated by the restriction of $\xi$ to $E$.
    By hypothesis, there is a linear embedding $\psi\colon E' \rightarrow F$ extending the identity such that $(1 - \varepsilon)\|y\| < \|\psi y\| < (1 + \varepsilon)\|y\|$ for all $y \in E'$, and in particular for $y = a_i + r_i x$, so $\tp(\psi x/F) \in U$.
  \item[\impprev]
    We prove this by induction on $n$, the case $n = 0$ being tautologically true.
    For the induction step, let $\emptyset \neq U \subseteq \tS_{\bar x,y}(F)$ be open, and let $V = U\rest_{\bar x} \subseteq \tS_{\bar x}(F)$.
    By \autoref{lem:VariableChange}, $V$ is open, and by the induction hypothesis there are $\bar b \in F^n$ such that $\tp(\bar b/F) \in V$.
    Now, consider the map $\theta\colon \tS_y(F) \rightarrow \tS_{\bar x,y}(F)$, sending $\tp(a/F) \mapsto \tp(\bar b,a/F)$.
    It is continuous (in fact, it is a topological embedding), so $\emptyset \neq \theta^{-1} U \subseteq \tS_1(F)$ is open.
    By hypothesis, there is $c \in F$ such that $\tp(c/F) \in \theta^{-1} U$, i.e., such that $\tp(\bar b,c/F) \in U$, as desired.
  \item[\impprev]
    Let $E \subseteq E'$ be finite-dimensional, with $E \subseteq F$, and let $\varepsilon > 0$.
    Let $\bar a$ be a basis for $E$, and let $\bar a,\bar b$ be a basis for $E'$, say $|\bar a| = n$ and $|\bar b| = m$.
    For $N \in \bN$, let $U_N \subseteq \tS_m(F)$ be defined by the (finitely many) conditions of the form $\| \sum s_i a_i + \sum r_j x_j \| \in (1 - \varepsilon, 1 + \varepsilon)$, where $s_i$ and $r_j$ are of the form $\frac{k}{N}$ and $\|\sum s_i a_i + \sum r_j b_j\| \in (1 - \varepsilon, 1 + \varepsilon)$.
    By hypothesis there is a tuple $\bar c \in F^m$ such that $\tp(\bar c/F) \in U_N$, and we may define $\psi\colon E' \rightarrow F$ being the identity on $E$ and sending $\bar b \mapsto \bar c$.
    For $N$ big enough, it follows from the construction that if $y \in E'$, $\|y\| = 1$, then $\bigl| \|\psi y\| - 1 \bigr| < 2\varepsilon$, which is good enough.
  \end{cycprf}
\end{proof}

Model theorists may find the second and third conditions of \autoref{lem:GurarijTarskiVaught} reminiscent of a topological formulation of the Tarski-Vaught Criterion: a metrically closed subset $A$ of a structure is an elementary substructure if and only if the set of types over $A$ realised in $A$ is dense.
Indeed,

\begin{thm}
  \label{thm:GurarijTheory}
  Let $T^*$ be the model completion of the theory of Banach spaces, as per \autoref{prp:BanachModelCompletion}.
  Then its models are exactly the Gurarij spaces.
  In particular, since $T^*$ is $\aleph_0$-categorical, there exists a unique separable Gurarij space (up to isometric isomorphism).
\end{thm}
\begin{proof}
  Let $E$ be a Banach space, and embed it in a model $F \vDash T^*$.
  Then, first, by quantifier elimination, $E$ is a model of $T^*$ if and only if $E \preceq F$.
  Second, by the topological Tarski-Vaught Criterion evoked above, $E \preceq F$ if and only if the set of types over $E$, in the sense of $\Th(F) = T^*$, realised in $E$, is dense.

  By \autoref{prp:BanachModelCompletion} the space of types over $E$ (in the sense of $T^* = \Th(E)$) is $\tS^{\leq 1}_n(E)$ as defined there.
  By a dilation argument, the set of types realised in $E$ is dense in $\tS_1(E)$ if and only if the set of types realised in $B(E)$ is dense in $\tS_1^{\leq 1}(E)$, and we conclude by \autoref{lem:GurarijTarskiVaught} (or, if one works with the whole space as an unbounded structure, the same holds without the dilation argument).
\end{proof}

As mentioned in the introduction, the isometric uniqueness of the separable Gurarij space was originally proved by Lusky \cite{Lusky:UniqueGurarij} using the Lazar-Lindenstrauss matrix representation of $L^1$ pre-duals.
The same was recently re-proved by Kubiś and Solecki \cite{Kubis-Solecki:GurariiUniqueness} using more elementary methods.
Upon careful reading, their argument essentially consists of showing that the separable Gurarij space is the Fraïssé limit of the class of finite-dimensional Banach spaces.
The first author points this out in \cite{BenYaacov:MetricFraisse} as an application of a general development of Fraïssé theory for metric structures (yielding yet another uniqueness proof).
From this point onward we shall leave continuous logic aside, and work entirely within the formalism of type spaces as introduced in \autoref{sec:QuantifierFreeTypes}.
As we shall see, uniqueness and existence of the separable Gurarij space also follow as easy corollaries from later results that do not depend explicitly on any form of formal logic (\autoref{cor:UniqueUniversalGurarij} and \autoref{lem:GurarijTypesCoMeagre}).

\begin{dfn}
  \label{dfn:AtomicBanachSpace}
  Let $E \subseteq F$ be Banach spaces.
  We say that $F$ is \emph{atomic} over $E$ if the type over $E$ of every finite tuple in $F$ is isolated.
\end{dfn}

By \autoref{prp:SeparableCategoricity}, every Banach space is atomic over $0$.

\begin{thm}
  \label{thm:AtomicGurarijEmbedding}
  Let $E \subseteq F_0 \subseteq F_1$ be Banach spaces with $\dim F_0/E$ finite and $F_1$ separable and atomic over $E$.
  Also let $\bG \supseteq E$ be a Gurarij space, and let $\varphi\colon F_0 \rightarrow \bG$ be an isometric embedding extending $\id_E$.
  Then for every $\varepsilon > 0$ there exist an isometric embedding $\psi\colon F_1 \rightarrow \bG$ extending $\id_E$ with $\| \psi\rest_{F_0} - \varphi \| < \varepsilon$.

  In particular, any separable Banach space atomic over $E$ embeds isometrically over $E$ in any Gurarij space containing $E$.
\end{thm}
\begin{proof}
  It is enough to prove this in the case where $\dim F_1/F_0 = 1$.
  We may then choose a basis $\bar a \in F_1^{n+1}$ for $F_1$ over $E$, such that in addition $a_0,\ldots,a_{n-1}$ generate $F_0$.
  By hypothesis, $\xi = \tp(\bar a/E) \in \tS_{n+1}(E)$ is isolated.
  Let $\rho\colon \tS_{n+1}(\bG) \rightarrow \tS_{n+1}(E)$ be the parameter restriction map, and let $K = \rho^{-1}(\xi)$, observing that for any $\varepsilon > 0$, we have that $B(K,\varepsilon) = \rho^{-1} B(\xi,\varepsilon)$ is a neighbourhood of $K$.
  Given $r > 0$, we construct a sequence of tuples $\bar c_k \in \bG^{n+1}$, each of which realises a type in $B(\xi,2^{-k}r)$, as follows.

  For $k = 0$, we let $V \subseteq \tS_{n+1}(\bG)$ be the set of semi-norms defined by $\|x_i - \varphi a_i\| < r$ for $i < n$, which is open and intersects $K$.
  Then $V \cap B(K,r)^\circ \neq \emptyset$ (where $\cdot^\circ$ denotes topological interior), and we choose $\bar c_0$ to realise some type there.
  Given $\bar c_k$, we let $U_k \subseteq \tS_{n+1}(\bG)$ be the set of semi-norms defined by $\|x_i - c_{k,i}\| < 2^{-k}r$ for $i \leq n$, which is again open and intersects $K$, and we choose $\bar c_{k+1}$ to realise a type in $U_k \cap B(K,2^{-n-1}r)^\circ$.

  We obtain a Cauchy sequence $(\bar c_k)$ converging to some $\bar c \in \bG^{n+1}$, whose type $\tp(\bar c/E)$, being the metric limit of $\tp(\bar c_k/E)$, must be $\xi$.
  Then the linear map $\psi\colon F_1 \rightarrow \bG$ that extends $\id_E$ by $a_i \mapsto c_i$ is an isometric embedding.

  Finally, reading through our construction, we have $\|\varphi a_i - c_i\| < 3r$ for all $i < n$, and choosing $r$ small enough, $\|\psi\rest_{F_0} - \varphi\|$ is as small as desired.
\end{proof}

In particular, any two separable Gurarij spaces atomic over $E$ embed in one another, but we can do better.

\begin{thm}
  \label{thm:AtomicGurarijIsomorphism}
  Let $\bG_i$ be separable Gurarij spaces atomic over $E$ for $i = 0,1$, and let $E \subseteq F \subseteq \bG_0$ with $\dim F/E$ finite.
  Also let $\varphi\colon F \rightarrow \bG_1$ be an isometric embedding extending $\id_E$.
  Then for any $\varepsilon > 0$ there exists an isometric isomorphism $\psi\colon \bG_0 \cong \bG_1$ extending $\id_E$ with $\|\psi\rest_F - \varphi\| < \varepsilon$.

  In particular, any two separable Gurarij spaces atomic over $E$ are isometrically isomorphic over $E$.
\end{thm}
\begin{proof}
  Follows from \autoref{thm:AtomicGurarijEmbedding} by a back-and-forth argument.
  Indeed, by induction on $n$, using \autoref{thm:AtomicGurarijEmbedding}, we construct finite dimensional subspaces $E \subseteq F_n \subseteq \bG_1$ and $E \subseteq F'_n \subseteq \bG_0$, as well as isometric embeddings $\varphi_n\colon F_n \hookrightarrow \bG_0$ and $\varphi'_n\colon F'_n \hookrightarrow \bG_1$ extending $\id_E$, such that:
  \begin{enumerate}
  \item $F_0 = F$ and $\varphi_0 = \varphi$.
  \item $F_n \subseteq F_{n+1}$, $F'_n \subseteq F'_{n+1}$.
  \item $\varphi_n(F_n) \subseteq F_n'$ and $\varphi'_n(F'_n) \subseteq F_{n+1}$.
  \item $\|\varphi'_n\varphi_n - \id_{F_n}\| + \|\varphi_{n+1}\varphi'_n - \id_{F'_n}\| < 2^{-n-1}\varepsilon$.
  \item $\overline{\bigcup F_n} = \bG_1$, $\overline{\bigcup F'_n} = \bG_0$.
  \end{enumerate}
  Once the construction is complete, we have
  \begin{gather*}
    \| \varphi_n - \varphi_{n+1}\| \leq \| \varphi_n - \varphi_{n+1} \varphi'_n \varphi_n\| + \| \varphi_{n+1} \varphi'_n \varphi_n - \varphi_{n+1}\| < 2^{-n-1} \varepsilon \|x\|.
  \end{gather*}
  It follows that the sequence $(\varphi_n)$ converges in norm to an isometric embedding $\psi\colon \bigcup F_n \hookrightarrow \bG_0$, which extends uniquely to $\bG_1 \hookrightarrow \bG_0$.
  We obtain $\psi'\colon \bG_0 \hookrightarrow \bG_1$ as a limit of $\varphi'_n$ similarly.
  Then $\psi$ extends $\id_E$, $\psi' = \psi^{-1}$, and $\|\psi\rest_F - \varphi\| < \varepsilon$, as desired.
\end{proof}

Since every Banach space is atomic over $0$, we obtain the uniqueness and universality of the separable Gurarij space.

\begin{cor}
  \label{cor:UniqueUniversalGurarij}
  Every two separable Gurarij spaces are isometrically isomorphic, and every separable Banach space embeds isometrically in any Gurarij space (separable or not).
\end{cor}

We also obtain that the Gurarij space is \emph{approximately homogeneous}.

\begin{cor}
  \label{cor:HomogeneousGurarij}
  Let $\bG$ be a separable Gurarij space, let $F \subseteq \bG$ be finite-dimensional, and let $\varphi\colon F \rightarrow \bG$ be an isometric embedding.
  Then there exists an isometric automorphism $\psi \in \Aut(\bG)$ such that $\|\psi\rest_F -\varphi\|$ is arbitrarily small.

  Moreover, if $E \subseteq F$ is such that $\bG$ is atomic over $E$, and $\varphi\rest_E = \id$, then we may require that $\psi\rest_E = \id$ as well.
\end{cor}

\begin{ntn}
  \label{ntn:Gurarij}
  We shall denote by $\bG$ the unique separable Gurarij space.
  Similarly, for a separable Banach space $E$, we let $\bG[E]$ denote the unique atomic separable Gurarij space over $E$, if such an extension of $E$ exists.
  Observe that since all types over $0$ are isolated, $\bG = \bG[0]$.
\end{ntn}

Let $E$ be a separable Banach space.
Let $\Emb(E,\bG)$ denote the space of linear isometric embeddings $E \hookrightarrow \bG$, on which $\Aut(\bG)$ acts by composition.
Say that $\varphi \in \Emb(E,\bG)$ is an \emph{atomic embedding} if $\bG$ is atomic over $\varphi E$.

\begin{cor}
  \label{cor:GenericOrbit}
  Let $E$ be a separable Banach space.
  Equip $\Emb(E,\bG)$ and $\Aut(\bG)$ with the topology of point-wise convergence (the strong operator topology).
  \begin{enumerate}
  \item The space $\Emb(E,\bG)$ is Polish, the action $\Aut(\bG) \curvearrowright \Emb(E,\bG)$ is continuous and all its orbits are dense.
  \item If $\bG[E]$ exists, then the set of atomic embeddings $\varphi \in \Emb(E,\bG)$ is a dense $G_\delta$ orbit under this action.
  \item If $\bG[E]$ does not exist, then there are no atomic embeddings and all orbits are meagre.
  \end{enumerate}
\end{cor}
\begin{proof}
  The first item is easy and left to the reader (density is by \autoref{cor:HomogeneousGurarij}).

  It follows from \autoref{thm:AtomicGurarijIsomorphism} that the set $Z \subseteq \Emb(E,\bG)$ of atomic embeddings forms a single orbit under $\Aut(\bG)$.
  By definition, $Z \neq \emptyset$ if and only if $\bG[E]$ exists.
  Let $\sI_n \subseteq \tS_n(E)$ denote the set of isolated types.
  For $r > 0$, we know that $B(\sI_n,r)$ is a neighbourhood of $\sI_n$, so there exists an open set $U_{n,r}$ such that $\sI_n \subseteq U_{n,r} \subseteq B(\sI_n,r)$ (in fact one can show that $B(\sI_n,r)$ is open, but we shall not require this).
  For each $\bar b \in \bG^n$, we define $V_{\bar b,r} \subseteq \Emb(E,\bG)$ to consist of all $\varphi$ such that $\tp(\bar b/\varphi E) \in \varphi U_{n,r}$.
  It is easy to see that since $U_{n,r}$ is open, so is $V_{\bar b,r}$.
  Since the set of isolated types is metrically closed, we have
  \begin{gather*}
    Z = \bigcap_{n,\bar b \in \bG^n, r>0} V_{\bar b,r} = \bigcap_{n, \bar b \in \bG_0^n,k} V_{\bar b, 2^{-k}},
  \end{gather*}
  where $\bG_0 \subseteq \bG$ is any countable dense subset.
  Thus, if $Z \neq \emptyset$ it is a dense $G_\delta$ orbit.

  Consider now a non atomic embedding $\psi \in \Emb(E,\bG)$.
  Non atomicity means that $\bG$ realises some non isolated type in $\tS_n(\psi E)$, which we may write as $\psi \xi$, where $\psi$ is applied to the parameters of $\xi$, and $\xi \in \tS_n(E)$ is non isolated.
  By \autoref{lem:IsolatedTypesMetricallyClosed}, for $r > 0$ small enough, the closed metric ball $\overline B(\xi,r)$ is (topologically) closed with empty interior.
  For $\bar b \in \bG^n$, let $V_{\bar b} \subseteq \Emb(E,\bG)$ consist of all $\varphi$ such that $\tp(\bar b/\varphi E) \notin \overline B(\varphi \xi,r)$.
  Reasoning as above, each $V_{\bar b}$ is a dense open set, and the set of $\varphi \in \Emb(E,\bG)$ such that $\bG$ omits $\varphi \xi$ is co-meagre.
  Since this set is also disjoint from the orbit of $\psi$, we are done.
\end{proof}

We now turn to a criterion for the existence of $\bG[E]$.
Say that a type $\xi \in \tS_\bN(E)$ is a \emph{Gurarij type} if $E[\xi]$, the generated space in the sense of \autoref{dfn:TypeOfTuple}, is a Gurarij space.
By an abuse of notation, we shall use $X = \{x_i\}_{i \in \bN}$ to denote both the set of variables and the set of distinguished generators of $E[\xi]$.

\begin{lem}
  \label{lem:GurarijTypesCoMeagre}
  Let $E$ be a separable Banach space.
  Then the set of Gurarij types over $E$ is co-meagre in $\tS_\bN(E)$.
  Moreover, there exists a dense $G_\delta$ set $Z \subseteq \tS_\bN(E)$ such that if some $\xi \in Z$ generates $F = E[\xi]$, then $F$ is Gurarij and the set of generators $\{x_i\}_{i \in \bN} \subseteq F$ is dense.

  In particular, the separable Gurarij space $\bG$ exists.
\end{lem}
\begin{proof}
  For $k \in \bN$, let $\{W_{k,m}\}_{m\in \bN}$ be a countable basis for the topology of $\tS_k(E)$, which exists since $E$ is separable.

  Whenever $I \subseteq J$ we shall use $\pi_{J,I}\colon \tS_J(E) \rightarrow \tS_I(E)$ to denote the variable restriction map, namely, $\tp\bigl( (a_i)_{i \in J} / E \bigr) \mapsto \tp\bigl( (a_i)_{i \in I} / E \bigr)$, which is an open map by \autoref{lem:VariableChange}.
  Thus, for example, $\bigl\{ \pi_{\bN,k}^{-1}(W_{k,m}) : k,m \in \bN \bigr\}$ is a basis of open sets for $\tS_\bN(E)$.

  Let us fix some $n \in \bN$, introduce yet another formal variable $y$, and consider an open set $U \subseteq \tS_{n+1}(E) = \tS_{x_0,\ldots,x_{n-1},y}(E)$.
  Let $F = \tS_n(E) \setminus \pi_{n+1,n}(U)$, which is closed.
  For each $k \in \bN$, let $\varphi_k^*\colon \tS_\bN(E) \rightarrow \tS_{n+1}(E)$ be the variable change map given by sending $x_i \rightarrow x_i$ for $i < n$ and $y \rightarrow x_k$, as per \autoref{dfn:VariableChange}, namely, $\tp(a_0,a_1,\ldots/E) \mapsto \tp(a_0,\ldots,a_{n-1},a_k/E)$.
  We then define
  \begin{gather*}
    \widetilde{U} = \pi_{\bN,n}^{-1}(F) \cup \bigcup_k (\varphi_k^*)^{-1}(U).
  \end{gather*}
  This is the union of an open and a closed set in a second-countable compact space, so it is $G_\delta$.

  First, we claim that if $k \geq n$ then $\pi_{\bN,k}(\widetilde{U}) = \tS_k(E)$.
  Indeed, we have a commutative diagram
  \begin{gather*}
    \begin{xy}
      (0,2)*+{\tS_\bN(E)}="SN",
      (2,0)*+{\tS_k(E)}="Sk",
      (2,4)*+{\tS_{n+1}(E)}="Sn1",
      (4,2)*+{\tS_n(E)}="Sn",
      \ar^{\pi_{\bN,k}} "SN";"Sk",
      \ar_{\varphi^*_k} "SN";"Sn1",
      \ar^{\pi_{k,n}} "Sk";"Sn"
      \ar_{\pi_{n+1,n}} "Sn1";"Sn"
      \ar_{\pi_{\bN,n}} "SN";"Sn",
    \end{xy}
  \end{gather*}
  Given any type $\xi \in \tS_k(E)$ there are two possibilities.
  \begin{itemize}
  \item If $\pi_{k,n}(\xi) \in F$ then $\pi_{\bN,k}^{-1}(\xi) \subseteq \pi_{\bN,n}^{-1}(F) \subseteq \widetilde{U}$, so $\xi \in \pi_{\bN,k}(\widetilde{U})$.
  \item If $\pi_{k,n}(\xi) \notin F$ then there exists $\zeta \in U$ such that $\pi_{n+1,n}(\zeta) = \pi_{k,n}(\xi) = \chi$, say.
    Amalgamating $E[\xi]$ with $E[\zeta]$ over $E[\chi]$, as per \autoref{fct:BanachSpaceAmalgamation}, we obtain a normed space $F \supseteq E$, a tuple $\bar b \in F^k$ such that $\tp(\bar b/E) = \xi$, and $c \in F$ such that $\tp(b_0,\ldots,b_{n-1},c) = \zeta$.
    Let $\rho = \tp(\bar b,c,c,c,\ldots/E) \in \tS_\bN(E)$.
    Then $\varphi_k^*(\rho) = \zeta$, and again $\xi = \pi_{\bN,k}(\rho) \in \pi_{\bN,k}(\widetilde{U})$.
  \end{itemize}

  Second, we claim that $\widetilde{U}$ is dense in $\tS_\bN(E)$.
  Indeed, consider a basic open set $\pi_{\bN,k}^{-1}(W_{k,m}) \neq \emptyset$.
  Since $\pi_{\bN,k}(\widetilde{U}) = \tS_k(E) \supseteq W_{k,m}$, we have $\emptyset \neq \pi_{\bN,k}^{-1}(W_k) \cap \widetilde{U}$.

  Lastly, we claim that the desired set is
  \begin{gather*}
    Z = \bigcap_{n,m} \widetilde{W_{n+1,m}} \subseteq \tS_\bN(E).
  \end{gather*}
  It is indeed a dense $G_\delta$ set.
  Let $\xi \in Z$, let $F = E[\xi]$ be the generated space, and $\bar b = (b_i)_{i\in \bN}$ be the generators.
  Then it will suffice to show that for any open $\emptyset \neq U \subseteq \tS_1(F)$ there exists $k$ such that $\tp(b_k/F) \in U$: this clearly implies that $\bar b$ is dense in $F$, and by \autoref{lem:GurarijTarskiVaught}, $F$ is Gurarij.
  Let us fix some $\rho \in U$, which we may always write as $\tp(c/F)$ where $c$ lies in some $F' \supseteq F$.

  We define $\theta\colon \tS_1(F) \rightarrow \tS_{Xy}(E)$ to be the map sending $\tp(a/F)$ to $\tp(\bar b,a/E)$, as in the proof of \autoref{lem:VariableChange} (working over $E$ instead of $0$).
  This is a topological embedding, so we may write $U = \theta^{-1}(W)$ with $W \subseteq \tS_{X,y}(E)$ open and $\theta(\rho) = \tp(\bar b,c/E) \in W$, and possibly shrinking $W$ (and $U$), we may assume that $W$ is defined using only finitely many variables, say $\bar x,y = x_0,\ldots,x_{n-1},y$.
  In other words, we may assume that $W = \pi_{Xy,\bar x y}^{-1}(W_{n+1,m})$ for some $m$, so $\pi_{Xy,\bar xy} \circ \theta(\rho) = \tp(b_0,\ldots,b_{n-1},c/E) \in W_{n+1,m}$.
  Thus $\rho$ provides us with a witness that
  \begin{gather*}
    \pi_{\bN,n}(\xi) = \tp(b_0,\ldots,b_{n-1}/E) \in \pi_{n+1,n}(W_{n+1,n}).
  \end{gather*}
  On the other hand, we have $\xi \in Z \subseteq \widetilde{W_{n+1,m}}$, so there must exist some $k$ such that
  \begin{gather*}
    \pi_{Xy,\bar xy} \circ \theta\bigl(\tp(b_k/F)\bigr) = \tp(b_0,\ldots,b_{n-1},b_k/E) = \varphi_k^*(\xi) \in W_{n+1,m}.
  \end{gather*}
  We conclude that $\tp(b_k/F) \in (\pi_{Xy,\bar xy} \circ \theta)^{-1}(W_{n+1,m}) = U$, and the proof is complete.
\end{proof}

\begin{thm}[Omitting Types Theorem for Gurarij spaces]
  \label{thm:GurarijOmittingTypes}
  Let $E$ be a separable Banach space, and assume we are given, for each $n \in \bN$, a metrically open and topologically meagre set $X_n \subseteq \tS_n(E)$.
  Then there exists a separable Gurarij space $\bG \supseteq E$ such that in addition, for every $n$, no type in $X_n$ is realised in $\bG$ (we then say that $\bG$ \emph{omits} all $X_n$).
  Moreover, the set of Gurarij types that generate such spaces is co-meagre.
\end{thm}
\begin{proof}
  Let $Z \subseteq \tS_\bN(E)$ be the set produced by \autoref{lem:GurarijTypesCoMeagre}.
  For each $n$, let $[\bN]^n = \{s \subseteq \bN\colon |s| = n\}$.
  Any $s \in [\bN]^n$ can be enumerated uniquely as an increasing sequence $\{k_0,\ldots, k_{n-1}\}$, and we then define $[s]\colon E(n) \rightarrow E(\bN)$ by $x_i \mapsto x_{k_i}$ for $i < n$.
  Then $[s]^*\colon \tS_\bN(E) \rightarrow \tS_n(E)$ is continuous and open, so $[s]_* X_n \subseteq \tS_\bN(E)$ is meagre.
  Since everything is countable,
  \begin{gather*}
    Z_1 = Z \setminus \bigcap_{n, s \in [\bN]^n} [s]_* X_n
  \end{gather*}
  is co-meagre as well.
  All we need to show is that if $\xi \in Z_1$ generates $\bG$, then $\bG$ omits $X_n$.
  Indeed, assume that some $\xi \in X_n$ is realised in $\bG$, say by $\bar a$.
  Since $X_n$ is metrically open, there exists $r > 0$ such that $B(\xi,r) \subseteq X_n$.
  Since the sequence $\{x_i\}$ is dense in $\bG$, and $\bG$ has no isolated points, there exists an increasing sequence $k_0 < \ldots < k_{n-1}$ such that $\|x_{k_j} - a_j\| < r$.
  But then $\tp(x_{\bar k}/E) \in X_n$, so $\xi \in [\bar k]_* X_n$, contradicting the choice of $\xi$ and completing the proof.
\end{proof}

Say that a Gurarij space $G$ is \emph{prime} over a separable subspace $E$ if it embeds isometrically over $E$ in every Gurarij space containing $E$.

\begin{cor}[Criterion for primeness over $E$]
  \label{cor:PrimenessCriterion}
  Let $G$ be a Gurarij space, and let $E \subseteq G$ be a separable subspace.
  Then the following are equivalent:
  \begin{enumerate}
  \item The space $G$ is prime over $E$.
  \item The space $G$ is separable and atomic over $E$, namely, $G = \bG[E]$.
  \end{enumerate}
\end{cor}
\begin{proof}
  If $G = \bG[E]$, then it is prime over $E$ by \autoref{thm:AtomicGurarijEmbedding}.
  For the other direction, assume that $\bG$ is prime over $E$.
  Since $E$ is separable, it embeds (by \autoref{thm:AtomicGurarijEmbedding}) in a separable Gurarij space, so $\bG$ must be separable as well.
  Finally, assume toward a contradiction that $\bG$ realises some non isolated type $\xi$.
  By \autoref{lem:IsolatedTypesMetricallyClosed} there exists $r > 0$ such that the closed metric ball $\overline B(\xi,r)$ has empty interior.
  Since the distance is lower semi-continuous, the closed metric ball is topologically closed, and is therefore meagre, as is the open ball $B(\xi,r)$.
  By \autoref{thm:GurarijOmittingTypes}, there exists a separable Gurarij space $\bG \supseteq E$ that omits $B(\xi,r)$.
  Thus $G$ cannot embed over $E$ in $\bG$, a contradiction.
\end{proof}

\begin{prp}
  \label{prp:AtomicGurarijExistence}
  Let $E$ be a separable Banach space.
  Then $\bG[E]$ exists if and only if, for each $n$, the set of isolated types in $\tS_n(E)$ is dense.
\end{prp}
\begin{proof}
  For a given $n$, let $\sI_n$ be the set of isolated types in $\tS_n(E)$, and assume that it is dense.
  Then $B(\sI_n,r)$ contains a dense open set, and $\bigcap_{r > 0} B(\sI_n,r) = \overline \sI_n$ is co-meagre.
  By \autoref{lem:IsolatedTypesMetricallyClosed} we have $\sI_n = \overline \sI_n$, so $\tS_n(E) \setminus \sI_n$ is meagre and metrically open.
  Therefore, by \autoref{thm:GurarijOmittingTypes}, if $\sI_n$ is dense for all $n$, then an atomic separable Gurarij space over $E$ exists.

  Conversely, assume that $\bG[E]$ exists.
  Then the set of $n$-types over $E$ realised in $\bG[E]$ is dense (by \autoref{lem:GurarijTarskiVaught}), and they are all isolated.
\end{proof}

Model theorists will recognise \autoref{prp:AtomicGurarijExistence} as the usual criterion for the existence of an atomic model, and as such it is in no way particular to Banach spaces.
In the specific context of Banach spaces, however, it can be improved, yielding \autoref{thm:AtomicGurarijExistence}.

\begin{lem}
  \label{lem:IsolatedTypeOneVariable}
  For a type $\xi \in \tS_{\bar x}(E)$ the following are equivalent
  \begin{enumerate}
  \item The type $\xi$ is isolated.
  \item The type $\xi \rest_{\bar y}$ is isolated for every $\bar y \in E(\bar x)^m$ (and every $m$).
  \item The type $\xi \rest_y$ is isolated for every $y \in E(\bar x)$.
  \end{enumerate}
\end{lem}
\begin{proof}
  \begin{cycprf}
  \item[\impnext]
    When $\bar y$ are linearly independent over $E$, this follows from \autoref{lem:VariableChange}.
    Hence, for the general case, it is enough to consider the situation where $\bar y$, of length $m$, extends the original tuple of variables $\bar x$, of length $n$.
    For $j < m$ let us write $y_j = a_j + \sum_{i<n} \lambda_{ij} x_i$.
    Given $r > 0$, there exists by hypothesis an open set $U$ such that $\xi \in U \subseteq B(\xi,r)$, and let $V = (\cdot\rest_{\bar x})^{-1} U \subseteq \tS_{\bar y}(E)$.
    Intersecting $V$ with the open sets defined by $\| y_j - \sum_{i<n} \lambda_{ij} y_i - a_j \| < r$ we obtain an open set $V'$ with $\xi \rest_{\bar y} \in V' \subseteq B(\xi \rest_{\bar y},r')$ for some $r' = r'(r,\bar y)$ that goes to zero with $r$.
  \item[\impnext]
    Immediate.
  \item[\impfirst]
    We repeat the proof of \autoref{prp:SeparableCategoricity} (in fact, that result is merely a special case of the present one, alongside the fact that types in $\tS_1(0)$ are trivially isolated).
    Indeed, for each $N$ there exists by hypothesis a neighbourhood $U_N \ni \xi$ consisting of $\zeta$ such that
    \begin{gather*}
      \forall y \in X_N \qquad d(\zeta \rest_y,\xi \rest_y) < 1/N.
    \end{gather*}
    Using \autoref{prp:TypeDistance} we conclude as for \autoref{prp:SeparableCategoricity}.
  \end{cycprf}
\end{proof}

\begin{thm}
  \label{thm:AtomicGurarijExistence}
  The following are equivalent for a separable Banach space $E$:
  \begin{enumerate}
  \item The space $\bG[E]$ exists.
  \item For each $n$, the set of isolated types in $\tS_n(E)$ is dense.
  \item The set of isolated types in $\tS_1(E)$ is dense.
  \end{enumerate}
\end{thm}
\begin{proof}
  We only need to show that if the set of isolated $1$-types is dense, then $\bG[E]$ exists.
  Indeed, proceeding as in the proof of \autoref{prp:AtomicGurarijExistence} there exists a separable Gurarij space $\bG \supseteq E$ that only realises isolated $1$-types over $E$.
  By \autoref{lem:IsolatedTypeOneVariable}, $\bG$ is atomic over $E$.
\end{proof}

\section{Isolated types over one-dimensional spaces}
\label{sec:IsolatedTypesDimensionOne}

Recall that one of the goals of this paper is to characterise isolated types over arbitrary $E$.
We start with the next-easiest case after $E = 0$, namely when $\dim E = 1$.
Even though this case will be fully subsumed in the general case, it is technically significantly simpler and deserves some specific comments, so we chose to treat it separately.

\begin{dfn}
  \label{dfn:NormingLinearFunctional}
  A \emph{norming linear functional} for $v \in E \setminus \{0\}$ is a continuous linear functional $\lambda \in E^*$ such that $\|\lambda\| = 1$ and $\lambda v = \|v\|$.

  By the Hahn-Banach Theorem, a norming linear functional always exists.
  We say that $v$ is \emph{smooth} in $E$ if the norming linear functional is unique.
\end{dfn}

\begin{prp}
  \label{prp:IsolatedTypesDimensionOne}
  Let $E$ be a Banach space, and let $v \in E \setminus \{0\}$.
  Then $E$ is atomic over $v$ if and only if $v$ is smooth in $E$.
\end{prp}
\begin{proof}
  By \autoref{lem:IsolatedTypeOneVariable}, we may assume that $E = \langle v, u \rangle$ and show that $\tp(u/v)$ is isolated if and only if $v$ is smooth in $E$.
  Assume first that for some $s,\varepsilon > 0$ and $D \in \bR$ we have
  \begin{gather*}
    \|v \pm su\| < \|v\| \pm sD + s\varepsilon.
  \end{gather*}
  It follows by the triangle inequality that
  \begin{gather*}
    \|v\| \pm tD - t\varepsilon \leq \|v \pm tu\| < \|v\| \pm tD + t\varepsilon, \qquad 0 < t \leq s,
  \end{gather*}
  or equivalently,
  \begin{gather*}
    \bigl| \|\pm rv + u\| - r\|v\| \mp D \bigr| < \varepsilon, \qquad r \geq s^{-1}.
  \end{gather*}

  If $v$ is smooth, let $\lambda$ be the unique norming functional, and let $D = \lambda u$.
  Then for any $\varepsilon > 0$ there exists $s$ as above.
  Then $\xi = \tp(u/v)$ satisfies the open condition $\|v \pm sx\| < \|v\| \pm sD + s\varepsilon$, which in turn implies that $\bigl| \|rv - u\| - \|rv - x\| \bigr| \leq 2\varepsilon$ for all $|r| \geq s^{-1}$.
  Finitely many additional open conditions can ensure that that the same holds for all $r$, yielding an open set $\xi \in U \subseteq B(\xi,3\varepsilon)$, showing that $\xi$ is isolated.

  Conversely, if $v$ is not smooth, then there are norming functionals $\lambda^\pm$, where $D^- = \lambda^- u < D^+ = \lambda^+ u$.
  Any neighbourhood of $\xi$ contains an open set $U$ that is defined by finitely many conditions of the form $\bigl| \|r_i v + x\| - \|r_i v + u\| \bigr| < \varepsilon$.
  We can construct a Banach space $E'$ generated by $\{v,w\}$, with $\|v\|$ as in $E$, such that $\zeta = \tp(w/v) \in U$ and $v$ is smooth in $E'$, with unique norming functional being defined by $\mu w = D^-$.
  This means that for $r$ big enough we have
  \begin{gather*}
    \|rv + w\| \approx r \|v\| + D^- \leq \|rv + u\| + D^- - D^+,
  \end{gather*}
  so $d(\xi,\zeta) \geq D^+ - D^-$.
  Therefore $B(\xi,D^+-D^-)$ is \emph{not} a topological neighbourhood of $\xi$, and $\xi$ is not isolated.
\end{proof}

We provided a fairly elementary argument to the ``only if'' part of \autoref{prp:IsolatedTypesDimensionOne}.
The machinery developed above provides us with a conceptually different argument, which in a sense we find preferable.
First, let us recall that by Mazur \cite[Satz~2]{Mazur:Konvexe}, the set of smooth points in the unit sphere of a separable Banach space is a dense $G_\delta$.
Assume now that $E$ is atomic over $v$, and without loss of generality, say that $\|v\| = 1$, and let $u \in \bG$ be smooth of norm one.
By \autoref{thm:AtomicGurarijEmbedding} there exists an isometric embedding of $E$ in $\bG$ sending $v$ to $u$, so $v$ must be smooth.

We obtain the following result of Lusky \cite{Lusky:UniqueGurarij}.

\begin{cor}
  The smooth points in the unit sphere of $\bG$ form a single dense $G_\delta$ orbit under isometric automorphisms.
\end{cor}
\begin{proof}
  Immediate from \autoref{prp:IsolatedTypesDimensionOne} and \autoref{thm:AtomicGurarijIsomorphism}.
\end{proof}

\begin{cor}
  The distance strictly refines the topology on $\tS_n(F)$ for every $F \neq 0$.
\end{cor}
\begin{proof}
  It follows from \autoref{lem:ParameterRestriction} that if $E \subseteq F$, and the topology and distance agreed on $\tS_n(F)$, then they would also agree on $\tS_n(E)$, and every type in $\tS_n(E)$ would be isolated.
  However, by \autoref{prp:IsolatedTypesDimensionOne}, not all types over a $1$-dimensional $E$ are isolated.
\end{proof}

\section{The Legendre-Fenchel transformation of $1$-types}
\label{sec:LegendreFenchel}

In this section we recall and develop a few technical tools that will be used later in order to characterise isolated types over arbitrary $E$.
We start with the \emph{Legendre-Fenchel transformation}.
This being a duality construction, it will be convenient for us to put $E$ and its dual $E^*$ on a more equal footing.

Recall that a \emph{locally convex} topology on a vector space (for our purposes, only over $\bR$), is a vector space topology admitting a basis of convex neighbourhoods for $0$.
Examples of such topologies include the norm topology on a normed space $E$, as well as the weak topology $w$ on $E$ and the weak$^*$ topology $w^*$ on $E^*$.
Moreover, $(E,w)$ and $(E^*,w^*)$ are one another's dual in the locally convex category, yielding the desired symmetry.

\begin{conv}
  In the rest of the paper, unless a more restrictive hypothesis is stated, $E$ will denote a locally convex topological vector space over $\bR$.
  Its \emph{topological dual} $E^*$ is the space of continuous linear functionals, always equipped with the \emph{weak$^*$ topology}, namely the least topology in which $\hat v\colon \lambda \mapsto \lambda v$ is continuous for each $v \in E$.

  This applies in particular when $E$ is a normed space: we may refer to the dual norm via the sets $B(E^*)$ and $\partial B(E^*)$, or the corresponding properties $\|\lambda\| \leq 1$ and $\|\lambda\| = 1$, but the topology on $E^*$ is always taken to be the weak$^*$ topology, so $B(E^*)$ is compact for any normed $E$.
\end{conv}

The weak$^*$ topology is again locally convex, and the bi-dual $E^{**}$ is canonically identified, as a set, with $E$ (which would not always be true if for a normed space $E$ we calculated $E^{**}$ with respect to the dual norm on $E^*$).
This induces the \emph{weak topology} on $E$, which may be weaker than the original one, but gives rise to the same dual $E^*$ ($= E^{***}$).
The Hahn-Banach Theorem tells us that the weak topology on $E$ agrees with the original one when it comes to closed convex sets.
In other words, both topologies give rise to the same notion of a closed convex function (defined below), which, for our purposes, is good enough.

\begin{fct}[Hahn-Banach Theorem, see Brezis \cite{Brezis:AnalyseFonctionnelle}]
  \label{fct:HahnBanach}
  A closed convex subset of $E$ (a locally convex vector space) is the intersection of the closed half-spaces that contain it, and is therefore weakly closed (a half-space of $E$ is a set of the form $\{v : \lambda v \leq r\}$ where $\lambda \in E^*$ and $r \in \bR$).
\end{fct}

\begin{dfn}
  \label{dfn:ProperClosedConvexFunction}
  Following Rockafellar \cite{Rockafellar:ConvexAnalysis}, a \emph{proper convex function} on $E$ is a convex function $f\colon E \rightarrow \bR \cup \{\infty\}$ that is not identically $\infty$.
  It is \emph{closed} if it is lower semi-continuous (equivalently, by \autoref{fct:HahnBanach}, lower semi-continuous in the weak topology).
  We then define its \emph{domain} $\dom f= \{v \in E\colon f(v) < \infty\}$.
  A \emph{closed convex function} is either a proper one or one of the constant functions $f = \pm \infty$ (an \emph{improper} one).

  For an arbitrary function $f\colon E \rightarrow [-\infty,+\infty]$ we define $f^*\colon E^* \rightarrow [-\infty,+\infty]$ by
  \begin{gather*}
    f^*(\lambda) = \sup_{v \in E} \lambda v - f(v).
  \end{gather*}
  If $f$ is closed convex we call $f^*$ its \emph{conjugate}.
\end{dfn}

\begin{fct}
  \label{fct:ConvexConjugate}
  For any $f\colon E \rightarrow [-\infty,+\infty]$, the function $f^*$ is closed convex.
  If $f$ is closed and convex, then $f = f^{**}$ under the canonical identification $E = E^{**}$, and $f^*$ is proper if and only if $f$ is.

  Moreover, if $g$ is another closed convex function, then $\|f-g\| = \|f^*-g^*\|$, where $\|{\cdot}\|$ denotes the supremum norm, possibly infinite, and we agree that $|\pm\infty \mp\infty| = 0$.
\end{fct}
\begin{proof}
  For the finite-dimensional case, see Rockafellar \cite[Section~12]{Rockafellar:ConvexAnalysis}.
  The general case is proved essentially in the same fashion, using \autoref{fct:HahnBanach}.
  The moreover part is easy to check directly.
\end{proof}

\begin{lem}
  \label{lem:GeneratedConvexFunction}
  Let $X \subseteq E$ and suppose $f\colon X \rightarrow \bR \cup \{\infty\}$ is not identically $\infty$.
  Assume moreover that whenever $x \in X$ can be expressed as a limit of convex combinations $\sum_{i<\ell_k} \, t_{k,i} x_{k,i}$, where $x_{k,i} \in X$, we have $f(x) \leq \liminf_k \, \sum_{i<\ell_k} \, f_{k,i} f(x_{k,i})$.
  Then extending $f$ by $\infty$ outside $X$, we have that $f^*$ is a proper closed convex function on $E^*$ and $f = f^{**}\rest_X$.
\end{lem}
\begin{proof}
  Let $\epi f = \bigl\{ (v,s) : f(v) \leq s \} \subseteq X \times \bR \subseteq E \times \bR$, the \emph{epigraph} of $f$, let $Y = \cco(\epi f) \subseteq E \times \bR$ be the closed convex hull and define $g(v) = \inf \{ t : (v,t) \in Y \} \in \bR \cup \{\infty\}$.
  Then $g$ is a closed proper convex function, $g \leq f$, and the hypothesis implies that $g$ agrees with $f$ on $X$.
  Now $g^* \geq f^*$, so $f^*$ is in particular proper (it is automatically closed and convex), and $g = g^{**} \leq f^{**} \leq f$.
  Therefore $f^{**}\rest_X = f$.
\end{proof}

We recall that if $X \subseteq E$ is convex, then $\partial X$ is defined as the set of all $v$ such that, for some affine line $L$, $v$ is one of two distinct boundary points of $L \cap X$ in $L$.
The \emph{relative interior}, sometimes denoted $\mathrm{ri}(X)$, is defined as $X \setminus \partial X$: the set of all $v \in X$ such that for every affine line $L$ going though $v$, either $L \cap X$ is a single point or contains $v$ in its interior relative to $L$.
When $X$ generates a finite-dimensional affine subspace, this agrees with the usual topological notions as calculated in that space.

\begin{lem}
  \label{lem:ConvexBoundaryContinuity}
  Let $E$ be finite-dimensional and let $X \subseteq E$ be a compact convex subset.
  Let $f\colon X \rightarrow \bR$ be closed and convex, and assume that $f\rest_{\partial X}$ is continuous.
  Then $f$ is continuous.
\end{lem}
\begin{proof}
  We need to show that if $x_n \rightarrow x$ in $X$ and $f(x_n) \rightarrow \alpha \in [-\infty,\infty]$, then $f(x) = \alpha$.
  Since $f$ is closed, $f(x) \leq \alpha$, and let us assume that $f(x) < \alpha$.
  We may then assume that $f(x_n) > f(x) + \varepsilon$ for some $\varepsilon > 0$ and all $n$.
  Then the ray $R_n = x_n + \bR^{\geq 0} (x_n-x)$ intersects $\partial X$ at a single point, call it $y_n = x_n + s_n(x_n-x)$, and we may further assume that $y_n \rightarrow y$, where $y$ is necessarily also on the boundary.
  Notice that $x_n = \cfrac{ s_n x + y_n }{s_n + 1}$, so by convexity $f(y_n) \geq (s_n+1) f(x_n) - s_n f(x) > f(x) + (s_n+1) \varepsilon \geq f(x) + \varepsilon$.
  Since $f$ is continuous on the boundary we must have $y \neq x$.
  But then $\|y_n - x\|$ is bounded away from zero, so $s_n \rightarrow \infty$, so $f$ is unbounded on the boundary, even though $\partial X$ is compact and $f$ is continuous there, a contradiction.
\end{proof}

The relevance of convex conjugation to our context comes from the following alternative characterisation of $1$-types over a normed space $E$, introduced in \cite{BenYaacov:UniversalGurarijIsometryGroup} (see also Katětov \cite{Katetov:UniversalMetricSpaces} and Uspenskij \cite{Uspenskij:SubgroupsOfMinimalTopologicalGroups}).
From now on, $E$ denotes a normed space.

\begin{dfn}
  \label{dfn:ConvexKatetov}
  Let $X$ be an arbitrary metric space.
  A \emph{Katětov function} on $X$ is a function $f\colon X \rightarrow \bR$ satisfying $f(x) \leq f(y) + d(x,y)$ and $d(x,y) \leq f(x) + f(y)$ for all $x,y \in X$.
  The space of Katětov functions on $X$ is denoted $K(X)$.
  As with type spaces, we equip $K(X)$ with a double structure, the topology of point-wise convergence and the distance of uniform convergence (i.e., the supremum distance).
  With this topology and distance, $K(X)$ is a topometric space (that is to say that the distance refines the topology, and is lower semi-continuous).

  If $X$ is a normed space, or a convex subset thereof, we let $K_C(X)$ denote the space of \emph{convex} Katětov functions on $X$, with the induced topometric structure.
\end{dfn}

\begin{fct}
  \label{fct:ConvexKatetovOneTypes}
  Let $\xi \in \tS_x(E)$ be a $1$-type over a normed space $E$, and let $f_\xi(a) = \|x-a\|^\xi$ for $a \in E$.
  Then
  \begin{enumerate}
  \item The map $\xi \mapsto f_\xi$ defines a bijection between $S_1(E)$ and $K_C(E)$, whose inverse is given by
    \begin{gather*}
      \|\alpha x - a\|^\xi =
      \begin{cases}
        \|a\| & \alpha = 0 \\
        |\alpha| f_\xi(a/\alpha) & \alpha \neq 0.
      \end{cases}
    \end{gather*}
  \item This bijection is a topological homeomorphism and a metric isometry.
  \end{enumerate}
\end{fct}
\begin{proof}
  The first item is \cite[Lemma~1.2]{BenYaacov:UniversalGurarijIsometryGroup}.
  For the second, that the bijection is homeomorphic (in the respective topologies of point-wise convergence) follows easily from the characterisation of the inverse, while the isometry is exactly \autoref{prp:TypeDistance} for $1$-types.
\end{proof}

Consequently, from now on we shall identify $K_C(E)$ with $\tS_1(E)$.

\begin{fct}
  \label{fct:ConvexKatatovExtension}
  Let $X \subseteq Y$ be metric spaces, and for $f \in K(X)$ and $y \in Y$ define
  \begin{gather*}
    \tilde f(y) = \inf_{x \in X} f(x) + d(x,y).
  \end{gather*}
  Then $\tilde f \in K(Y)$ extends $f$, and the induced embedding $K(X) \subseteq K(Y)$ is isometric.
  When $Y = E$ is a normed space, $X \subseteq E$ is convex and $f \in K_C(X)$, the extension $\tilde f$ is convex as well, inducing an isometric embedding $K_C(X) \subseteq K_C(E)$.
\end{fct}
\begin{proof}
  The first assertion goes back to Katětov \cite{Katetov:UniversalMetricSpaces}, and the second is \cite[Lemma~1.3(i)]{BenYaacov:UniversalGurarijIsometryGroup}.
\end{proof}

\begin{qst}
  If $X \subseteq E$ is convex and compact (or totally bounded), then the topology and distance agree on $K_C(X)$, and it follows that the inclusion $K_C(X) \subseteq K_C(E)$ is also continuous, and therefore homeomorphic (since the restriction map is always continuous).
  At the other extremity, if $X = E$, then the inclusion is homeomorphic as well.
  What about general convex $X \subseteq E$?
\end{qst}

A closed proper convex function $f\colon E \rightarrow \bR \cup \{\infty\}$ is essentially the same thing as a closed convex function $f\colon X \rightarrow \bR$, with convex domain $X$, such that $\liminf_{v \rightarrow u} f(v) = \infty$ for all $u \in \overline X \setminus X$.
Indeed, we can get one from the other by restricting to the finite domain in one direction, or by extending by $\infty$ in the other.
A special case of the second form is when $X \subseteq E$ is closed and convex and $f \in K_C(X)$.
If $X$ is merely convex, every $f \in K_C(X)$, being $1$-Lipschitz, admits a unique extension to $\overline f \in K_C(\overline X)$, so requiring $X$ to be closed is not truly a constraint.

\begin{dfn}
  \label{dfn:Antipode}
  Let $f\colon E \rightarrow \bR \cup \{\infty\}$ be a proper closed convex function, and let $\lambda \in E^*$.
  \begin{itemize}
  \item We shall say that $f$ (or more precisely, $f^*$) satisfies the \emph{antipode inequality} at $\lambda$ if $f^*(\lambda) + f^*(-\lambda) \leq 0$.
  \item It satisfies the \emph{antipode identity} at $\lambda$ if $f^*(\lambda) + f^*(-\lambda) = 0$.
  \end{itemize}
\end{dfn}

\begin{lem}
  \label{lem:ConvexKatetovConjugate}
  Let $X \subseteq E$ be closed and convex and let $f \in K_C(X)$.
  Then
  \begin{enumerate}
  \item
    \label{item:ConvexKatetovConjugateDomain}
    The domain $\dom f^*$ contains $B(E^*)$, and if $\lambda \in \dom f^*$ with $\|\lambda\| > 1$, then
    \begin{gather*}
      f^*(\lambda) = \sup_{v \in \partial X} \lambda v - f(v).
    \end{gather*}
    In particular, if $X = E$ (so $\partial X = \emptyset$), then $\dom f^*$ is exactly the closed unit ball.
  \item
    \label{item:ConvexKatetovConjugateBoundedDomain}
    If $X$ is bounded and $\|\lambda\| = 1$, then $f^*(\lambda) = \sup_{v \in \partial X} \lambda v - f(v)$.
  \item
    \label{item:ConvexKatetovConjugateExtension}
    Let $\tilde f \in K_C(E)$ be as per \autoref{fct:ConvexKatatovExtension}.
    Then
    \begin{gather*}
      \tilde f^*(\lambda) =
      \begin{cases}
        f^*(\lambda) & \|\lambda\| \leq 1 \\
        \infty & \|\lambda\| > 1
      \end{cases}
      \quad \text{and} \quad
      \tilde f(v) = \sup_{\|\lambda\| \leq 1} \lambda v - f^*(\lambda).
    \end{gather*}
    In addition, if $v \notin X$, then $\tilde f(v) = \sup_{\|\lambda\| = 1} \lambda v - f^*(\lambda)$.
  \item
    \label{item:ConvexKatetovConjugateHB}
    When $X = E$, we have $f \in K_C(E) = \tS_1(E)$.
    For $\lambda \in \partial B(E^*)$, the least possible value of a norm-preserving extension of $\lambda$ at a realisation of $f$ is $f^*(\lambda)$.
  \item
    \label{item:ConvexKatetovConjugateCriterion}
    Let $g\colon E \rightarrow \bR \cup \{\infty\}$ be any closed proper convex function.
    Then $g \in K_C(E)$ if and only if $\dom g^* = B(E^*)$ and the antipode inequality $g^*(\lambda) + g^*(-\lambda) \leq 0$ holds for all $\lambda \in \partial B(E^*)$, or, equivalently, for all $\lambda \in B(E^*)$.
  \item
    \label{item:ConvexKatetovAntipodeIdentity}
    Assume $g \in K_C(E)$ is such that the antipode identity $g^*(\lambda) + g^*(-\lambda) = 0$ holds at some $\lambda \in B(E^*)$.
    Then $g^*$ is continuous at $\lambda$.
  \end{enumerate}
\end{lem}
\begin{proof}
  For \autoref{item:ConvexKatetovConjugateDomain}, first let $\|\lambda\| \leq 1$, and let $u \in X$ be fixed.
  Then for all $v \in X$ we have $f(u) + \|u\| \geq \|v - u\| - f(v) + \|u\| \geq \lambda v - f(v)$, whereby $f^*(\lambda) \leq f(u) + \|u\| < \infty$.
  Now let $\|\lambda\| > 1$, say $\lambda w > \|w\|$, and assume that $\lambda \in \dom f^*$.
  Then for each $v \in \dom f$, the ray $\{v + \alpha w\colon \alpha \geq 0\}$ cannot be contained in $X$ (or else $f^*(\lambda) = \infty$) and therefore intersects the boundary, say at $v'$.
  In this case $\lambda v - f(v) \leq \lambda v' - f(v')$, proving our assertion.
  When $\|\lambda\| = 1$ but $X$ is assumed to be bounded, for every $v \in X$ we can find $v' \in \partial X$ with $\lambda v - f(v) \leq \lambda v' - f(v') + \varepsilon$ for $\varepsilon$ arbitrarily small, whence \autoref{item:ConvexKatetovConjugateBoundedDomain}.

  For \autoref{item:ConvexKatetovConjugateExtension}, we already know that $\dom \tilde f^*$ is exactly the closed unit ball.
  In addition, $\tilde f \leq f$ implies $\tilde f^* \geq f^*$, and if $\|\lambda\| \leq 1$, then for every $v \in E$ and $u \in X$:
  \begin{gather*}
    \lambda v - \tilde f(v) = \lambda v - \inf_{u \in X} \bigl[ f(u) + \|v-u\| \bigr] \leq \sup_{u \in X} \lambda u - f(u) = f^*(\lambda),
  \end{gather*}
  whereby $\tilde f^*(\lambda) \leq f^*(\lambda)$.
  This gives us the first identity, and then \autoref{fct:ConvexConjugate} gives the second one.
  Now assume that $v \notin X$, so by \autoref{fct:HahnBanach} there exists $\mu \in \partial B(E^*)$ such that $\mu\rest_X < \mu v$.
  For any $\lambda \in B(E^*)$ there exists $\alpha \geq 0$ such that $\|\lambda + \alpha \mu\| = 1$.
  Then $f^*(\lambda + \alpha \mu) \leq f^*(\lambda) + \alpha \mu v$, or equivalently, $\lambda v - f^*(\lambda) \leq (\lambda + \alpha \mu)v - f^*(\lambda + \alpha \mu)$, whence it follows that $\tilde f(v) = \sup_{\|\lambda\| = 1} \lambda v - f^*(\lambda)$.

  Item \autoref{item:ConvexKatetovConjugateHB} is immediate.
  For \autoref{item:ConvexKatetovConjugateCriterion}, we have already seen that if $g \in K_C(E)$, then $\dom g^* = B(E^*)$, and the previous item implies that $g^*(\lambda) + g^*(-\lambda) \leq 0$ for $\lambda \in \partial B(E^*)$.
  Conversely, assume that $\dom g^* = B(E^*)$ and the antipode inequality holds.
  Then $g = g^{**}$ is necessarily $1$-Lipschitz, so $\dom g = E$.
  For distinct $v,u \in E$, let $\lambda \in \partial B(E^*)$ norm $v - u$.
  Then
  \begin{gather*}
    g(v) + g(u) \geq \lambda v - g^*(\lambda) - \lambda u - g^*(-\lambda) \geq \lambda(v-u) = \|v-u\|,
  \end{gather*}
  as desired.

  For \autoref{item:ConvexKatetovAntipodeIdentity}, let $\lambda_\alpha \rightarrow \lambda$.
  Then $g^*(\lambda_\alpha) \leq -g^*(-\lambda_\alpha)$ by the antipode inequality and $g^*(\lambda) = -g^*(-\lambda)$ by hypothesis.
  Since $g^*$ is lower semi-continuous,
  \begin{align*}
    g^*(\lambda)
    & \leq \liminf g^*(\lambda_\alpha)
      \leq \limsup g^*(\lambda_\alpha)
      \leq \limsup -g^*(-\lambda_\alpha)
    \\
    &
      = -\liminf g^*(-\lambda_\alpha)
      \leq -g^*(-\lambda) = g^*(\lambda).
  \end{align*}
  Therefore $\lim g^*(\lambda_\alpha) = g^*(\lambda)$, as desired.
\end{proof}

\begin{rmk}
  \label{rmk:ConvexKatetovConjugateExtensionSubSpace}
  Let $F \subseteq E$ be normed spaces and let $g \in K_C(F)$.
  Since $F$ is convex in $E$, there may be some ambiguity about $g^*$, so let $g^*_F$ denote the conjugate as a convex function on $F$ and let $g^*_E$ denote the conjugate of the extension by infinity to $E^*$.
  Also let $\tilde g \in K_C(E)$ denote the canonical extension of $g$.
  Then $g^*_E(\lambda) = g^*_F(\lambda\rest_F)$ for $\lambda \in E^*$, and by \autoref{lem:ConvexKatetovConjugate}\autoref{item:ConvexKatetovConjugateExtension}, if $\|\lambda\| \leq 1$, then this is further equal to $\tilde g^*(\lambda)$.
  Therefore, in what interests us, this ambiguity can never lead to any form of confusion.
\end{rmk}

We obtain a characterisation of the realised types.

\begin{lem}
  \label{lem:AntipodeIdentityRealisedType}
  Let $E$ be a Banach space, $f \in K_C(E)$.
  Then the following are equivalent:
  \begin{enumerate}
  \item The type $f$ is realised in $E$, i.e., there exists $v \in E$ such that $f(x) = \|x-v\|$ for all $x \in E$.
  \item The conjugate $f^*$ satisfies the antipode identity throughout $B(E^*)$.
  \item The conjugate $f^*$ satisfies the antipode identity at some $\lambda \in B(E^*) \setminus \partial B(E^*)$.
  \item We have $f^*(0) = 0$.
  \end{enumerate}
\end{lem}
\begin{proof}
  \begin{cycprf}
  \item We have $f^*(\lambda) = \lambda v$.
  \item Clear.
  \item Assume $f^*(0) \neq 0$.
    Then necessarily $f^*(0) < 0$ and $\lambda \neq 0$.
    Let $\alpha = \|\lambda\| < 1$ and $\mu = \lambda / \alpha$, so $\lambda = \alpha \mu + (1-\alpha) 0$.
    By convexity, $\alpha f^*(\mu) + \alpha f^*(-\mu) + 2(1-\alpha) f^*(0) \geq f^*(\lambda) + f^*(-\lambda)  = 0$, contradicting the antipode inequality at $\mu$.
  \item[\impfirst]
    We have $0 = f^*(0) = - \inf f$.
    Since $f$ is Katětov, any sequence $v_k$ such that $f(v_k) \rightarrow 0$ must be Cauchy, say with limit $v$.
    It follows that $f(v) = 0$ and consequently that $f$ is realised by $v$.
  \end{cycprf}
\end{proof}

\section{Characterising isolated types over arbitrary spaces}
\label{sec:IsolatedTypes}

In \cite{BenYaacov:UniversalGurarijIsometryGroup} a special kind of ``well behaved'' convex Katětov functions is distinguished.
These will play a crucial role here as well, and admit a natural characterisation in terms of their conjugates.

\begin{dfn}
  \label{dfn:ConvexKatetovLocal}
  We say that a function $f \in K_C(E)$ is \emph{local} if there are $f_k \in K_C(X_k)$, where each $X_k \subseteq E$ is convex and compact, such that $\tilde f_k \rightarrow f$ uniformly.
  The set of local functions in $K_C(E)$ was denoted in \cite{BenYaacov:UniversalGurarijIsometryGroup} by $K_{C,0}(E)$.
\end{dfn}

\begin{lem}
  \label{lem:ConvexKatetovLocal}
  Let $E$ be a normed space, and let $f \in K_C(E)$.
  Then $f$ is local if and only if $f^*$ is continuous on $B(E^*)$.
\end{lem}
\begin{proof}
  First let $X \subseteq E$ be compact and let $g \in K_C(X)$.
  If $X \subseteq \bigcup_{i<n}  B(v_i,r)$, then $g^*(\lambda) - g^*(\mu) \leq 2r\| \lambda-\mu\| + \max_i (\lambda-\mu)v_i$, whence it follows that $g^*$ is continuous on every bounded subset of $E^*$, and in particular on $B(E^*)$.
  Since a uniform limit of continuous functions is continuous, if $f$ is local, then $f^*$ is continuous on $B(E^*)$.

  Conversely, assume that $f^*$ is continuous on $B(E^*)$.
  Since $\bR$ is second countable, there exists a separable subspace $F \subseteq E$ such that, for $\lambda \in B(E^*)$, the value of $f^*(\lambda)$ only depends on $\lambda\rest_F$.
  In this case, the map $f'\colon B(F^*) \rightarrow \bR$, $\mu \mapsto f^*(\mu')$, where $\mu'$ is any norm-preserving extension of $\mu$, is continuous.
  By \autoref{fct:ConvexConjugate} there exists a proper closed convex function $g\colon F \rightarrow \bR \cup \{\infty\}$ such that $g^* = f'$.
  By \autoref{lem:ConvexKatetovConjugate} we then have $g \in K_C(F)$ and by \autoref{rmk:ConvexKatetovConjugateExtensionSubSpace} we have $f^* = \tilde g^*$, so $f = \tilde g$.
  We may therefore assume that $E$ is separable, and choose an increasing sequence of compact convex subsets $X_k \subseteq E$ such that $\bigcup X_k$ is dense in $E$ (take closed balls of increasing radius, of sub-spaces of increasing finite dimension).
  For each $k$ let $f_k = f\rest_{X_k}$.
  Then $\tilde f_k \searrow f$ point-wise, and for $\lambda \in B(E^*)$ we have
  \begin{gather*}
    f^*(\lambda) = \sup_v \lambda v - f(v) = \sup_{v,k} \lambda v - f_k(v) = \sup_k f_k^*(\lambda),
  \end{gather*}
  i.e., $f_k^* \nearrow f^*$ point-wise on $B(E^*)$.
  Since each $f_k^*$ is lower semi-continuous, $f^*$ is continuous, and $B(E^*)$ is compact, this implies that $f_k^* \rightarrow f^*$ uniformly on $B(E^*)$, whereby $\tilde f_k \rightarrow f$ uniformly, and $f$ is local.
\end{proof}

A second ingredient is the following.

\begin{dfn}
  \label{dfn:BoundaryExtreme}
  Let $E$ be a normed space and let $r \geq 0$.
  We say that $f \in K_C(E)$ is \emph{$\partial$-$r$-extreme} (where $\partial$ could be pronounced \emph{boundary}) if for every $g \in K_C(E)$ whenever $g \leq f$ (i.e., $g^* \geq f^*$) we have $g^* \leq f^* + r$ on $\partial B(E^*)$.
  If $r = 0$ we omit it and say that $f$ is \emph{$\partial$-extreme}.
\end{dfn}

Notice that \autoref{dfn:BoundaryExtreme} would remain unchanged if we replaced the requirement that $g^* \geq f^*$ on $B(E^*)$ with $g^* \geq f^*$ on $\partial B(E^*)$.
Indeed, assume \autoref{dfn:BoundaryExtreme} holds of $f$ and $g^* \geq f^*$ on $\partial B(E^*)$.
Then $h' = \max(f^*, g^*)$ agrees with $g^*$ on $\partial B(E^*)$ and is therefore of the form $h^*$ for some $h \in K_C(E)$, by \autoref{lem:ConvexKatetovConjugate}\autoref{item:ConvexKatetovConjugateCriterion}, so \autoref{dfn:BoundaryExtreme} applies to $h$ and yields $g^* = h^* \leq f^* + r$ on $\partial B(E^*)$.

\begin{lem}
  \label{lem:BoundaryExtreme}
  Let $f,g \in K_C(E)$ with $f$ $\partial$-$r$-extreme and $g \leq f$.
  Then $g$ is $\partial$-$r$-extreme as well.

  Assume furthermore that $f = \widetilde{f\rest_X}$ for some convex $X \subseteq E$.
  Then outside $X$ we have $g \geq f - r$.
\end{lem}
\begin{proof}
  By hypothesis we have $g^* \geq f^*$, and we clearly obtain the first assertion, as well as $g^*(\lambda) \leq f^*(\lambda) + r$ for $\|\lambda\| = 1$.
  By \autoref{lem:ConvexKatetovConjugate}\autoref{item:ConvexKatetovConjugateExtension}, if $v \notin X$, then
  \begin{gather*}
    \tilde f(v)
    = \sup_{\|\lambda\| = 1} \, \lambda v - f^*(\lambda)
    \leq \sup_{\|\lambda\| = 1} \, \lambda v - g^*(\lambda) + r
    \leq g(v) + r,
  \end{gather*}
  as claimed.
\end{proof}

\begin{lem}
  \label{lem:BoundaryExtremeWithError}
  Let $E$ be a normed space, let $f \in K_C(E)$ be $\partial$-$r$-extreme, and let $\delta > 0$.
  Then $f + \delta$ is $\partial$-$(r+2\delta)$-extreme.
\end{lem}
\begin{proof}
  Assume not, so let $g \in K_C(E)$ satisfy $g \leq f + \delta$ (i.e., $g^* \geq f^* - \delta$), and let $\lambda \in \partial B(E^*)$ be such that $g^*(\lambda) > f^*(\lambda) + r + \delta$.
  By \autoref{fct:ConvexConjugate} there exists a closed convex $h$ on $E$ such that $h^* = \max \bigl(f^*, (g^*-\delta)\bigr)$.
  For $\mu \in B(E^*)$ we have $f^*(\mu) + f^*(-\mu) \leq 0$, $g^*(\mu) - \delta + g^*(-\mu) - \delta \leq -2\delta < 0$ and $f^*(\mu) + g^*(-\mu) - \delta \leq g^*(\mu) + f^*(-\mu) \leq 0$.
  Thus $h^*(\mu) + h^*(-\mu) \leq 0$, and since the domain of $h^*$ is the unit ball, $h \in K_C(E)$.
  Now, $h^* \geq f^*$ implies $h \leq f$, while on the other hand $h^*(\lambda) \geq g^*(\lambda) - \delta > f^*(\lambda) + r$, witnessing that $f$ is not $\partial$-$r$-extreme.
\end{proof}

\begin{lem}
  \label{lem:BoundaryExtremeNeighbourhood}
  Let $E$ be a normed space, let $f \in K_C(E)$ be $\partial$-$r$-extreme and local, and let $r' > r$.
  Then $f$ admits a neighbourhood $f \in W \subseteq K_C(E)$ such that $\diam W < r'$ and every $g \in W$ is $\partial$-$r'$-extreme.
\end{lem}
\begin{proof}
  Let $\delta > 0$ be small enough.
  By locality, there exists a compact convex set $X \subseteq E$ such that $f + \delta > f' = \widetilde{f\rest_X}$.
  Let $W \subseteq K_C(E)$ consist of all $g$ such that $|f-g| < \delta$ on $X$.
  Since $X$ is compact and Katětov functions are $1$-Lipschitz, $W$ is open, and we claim that it is as desired.

  If $g \in W$, then $g < f' < f + \delta$.
  By \autoref{lem:BoundaryExtremeWithError}, $f + \delta$ is $\partial$-$(r + 2\delta)$-extreme.
  By \autoref{lem:BoundaryExtreme} so are $f'$ and $g$, and since $f' = \widetilde{f'\rest_X}$ we have $g \geq f' - r - 2\delta \geq f - r - 2\delta$ outside $X$.
  Since $g \geq f-\delta$ inside $X$, we conclude that $f - r - 2\delta \leq g < f + \delta$ throughout, which is enough.
\end{proof}

\begin{thm}
  \label{thm:IsolatedTypeCharacterisation}
  Let $f \in K_C(E)$.
  Then $f$ is isolated if and only if it is both local and $\partial$-extreme.
\end{thm}
\begin{proof}
  One direction follows directly from \autoref{lem:BoundaryExtremeNeighbourhood}, so let us assume that $f$ is isolated.
  We can then construct a sequence of neighbourhoods $W_k$ of $f$ such that $\diam W_k \rightarrow 0$, each defined using finitely many parameters.
  We let $X_k$ be the (compact) convex hull of these parameters and $f_k = f\rest_{X_k}$.
  Then $\tilde f_k \in W_k$, so $\tilde f_k \rightarrow f$ uniformly and $f$ is local.

  It remains to show that $f$ is $\partial$-extreme, so let $g \in K_C(E)$ satisfy $g \leq f$.
  Let $W$ be a neighbourhood of $f$ of small diameter, say $g \in W$ if and only if $|g(v_i)-f(v_i)| < \varepsilon$ for some $v_i$, $i < n$.
  We know that $f(v) = \sup_{\|\lambda\| \leq 1} \, \lambda v - f^*(\lambda)$, and since a closed convex function in dimension one is continuous on its domain, we have in fact $f(v) = \sup_{\|\lambda\| < 1} \, \lambda v - f^*(\lambda)$.
  Therefore, for each $i < n$ we may choose $\lambda_i \in E^*_{<1}$ such that $f(v_i) - \varepsilon < \lambda_i v_i - f^*(\lambda_i) \leq f(v_i)$.
  Let $g' = \max\biggl(g, \max_{i<n} \, \bigl( \lambda_i - f^*(\lambda_i) \bigr)\biggr)$.
  Since $\|\lambda_i\| < 1$ for each $i$, $g'$ agrees with $g$ outside some ball.
  On the other hand, we have $g'(v_i) > f(v_i) - \varepsilon$ and $g' \leq f$, so $g' \in W$.
  Therefore $|f-g| \leq \diam W$ outside some ball.
  Since $\diam W$ can be taken arbitrarily small, $\lim_{\|x\| \rightarrow \infty} f(x)-g(x) = 0$.
  It follows that $f^* = g^*$ on $\partial B(E^*)$.
\end{proof}

Before stating a few more corollaries, let us recall a few definitions and facts.
\begin{dfn}
  Let $E$ be a locally convex space and let $X \subseteq E$ be convex.
  \begin{enumerate}
  \item A convex subset $F \subseteq X$ is called a \emph{face} of $X$ if a member of $F$ cannot be expressed as a proper convex combination of two points in $X$ that are not both in $F$.
    A \emph{proper face}, i.e., a face $F \neq X$, is always contained in the relative boundary $\partial X$.
  \item A face consisting of a single point is called an \emph{extreme point}.
    We shall denote the set of extreme points of $X$ by $\cE(X)$.
    We shall also denote by $\cE_0(X)$ the set of $v \in \cE(X)$ such that $\lambda v = \sup \, \lambda\rest_X$ for some $\lambda \in E^* \setminus \{0\}$.
  \end{enumerate}
\end{dfn}

By the Krein-Milman Theorem \cite[Chapitre~II.7, Théorème~1]{Bourbaki:EspacesVectorielsTopologiques}, if $X$ is compact and convex, then $X = \cco\bigl( \cE(X) \bigr)$.
In addition, by \cite[Chapitre~II.7, Proposition~2]{Bourbaki:EspacesVectorielsTopologiques}, if $v \in \cE(X)$, then the family of sets $\{u \in X : \lambda u > r\}$, where $\lambda \in E^*$ and $\lambda v > r \in \bR$, forms a basis of neighbourhoods for $v$ in $X$.
Since every such neighbourhood contains a member of $\cE_0(X)$ (any extreme point of $\{u \in X : \lambda u = \sup \, \lambda\rest_X\}$ will do), we have $\cE(X) \subseteq \overline{\cE_0(X)}$.
In the special case where $X = B(E^*)$, the set $\cE_0\bigl( B(E^*) \bigr)$ consists exactly of those $\lambda \in \cE\bigl( B(E^*) \bigr)$ for which some vector $v \neq 0$ is normed by $\lambda$.

\begin{cor}
  \label{cor:IsolatedTypeNecessaryCondition}
  Let $E$ be a Banach space, $f \in K_C(E)$ isolated.
  Let $F_v = \{\lambda \in B(E^*) : \lambda v = 1\}$ where $\|v\| = 1$.
  Then $f^*\rest_{F_v}$ is the greatest closed convex function less than $\lambda \mapsto -f^*(\lambda)$, i.e., $f^*(\lambda) = \sup_{u \in E} \, \inf_{\mu \in F_v} \, (\lambda-\mu) u - f^*(-\mu)$ for $\lambda \in F_v$.

  In particular, the antipode identity $f^*(\lambda) + f^*(-\lambda) = 0$ holds at every $\lambda \in \cE\bigl( B(E^*) \bigr)$.
\end{cor}
\begin{proof}
  For $u \in E$ and $\alpha > 0$ define
  \begin{gather*}
    h_u(\lambda) = \inf_{\mu \in F_v} \, (\lambda - \mu)u - f^*(-\mu),
    \qquad
    h_{u,\alpha}(\lambda) = \alpha(\lambda v - 1)  + h_u(\lambda).
  \end{gather*}
  Since $F_v$ is closed, applying \autoref{fct:ConvexConjugate} we have $f^*(\lambda) = \sup_{u \in E} \, \inf_{\mu \in F_v} \, (\lambda-\mu) u + f^*(\mu) \leq \sup_{u \in E} \, h_u(\lambda)$.
  For the converse inequality it will suffice to show that $f^* \geq h_u$ on $F_v$.
  Notice that $h_u$ is linear and continuous, and in addition, for every $\lambda \in F_v$ we have $h(\lambda) + f^*(-\lambda) \leq 0$.
  Fixing $\varepsilon > 0$, there exists an open set $V \supseteq F_v$ such that $h(\lambda) + f^*(-\lambda) < \varepsilon$ for all $\lambda \in V \cap B(E^*)$.
  By compactness of $B(E^*) \setminus V$, we know that $\lambda \mapsto \lambda v$ is bounded there below some $r < 1$.
  We may therefore assume that $V = \{\lambda : \lambda v > r\}$, and that $r > 0$, so $V \cap -V = \emptyset$.
  For $\alpha$ big enough we have $h_{u,\alpha}(\lambda) \leq \inf f^*$ for all $\lambda \in B(E^*) \setminus V$.
  Having fixed such $\alpha$, for $\lambda \in B(E^*) \cap V$ we have $h_{u,\alpha}(\lambda) - \varepsilon + f^*(-\lambda) \leq h_u(\lambda) + f^*(-\lambda) - \varepsilon \leq 0$.
  It follows that $\max \bigl( f^*, (h_{u,\alpha}-\varepsilon) \bigr)$ satisfies the antipode inequality, and is therefore of the form $g^*$ for some $g \in K_C(E)$.
  But then $g \leq f$, so $g^* = f^*$ on $\partial B(E^*)$ and in particular on $F_v$.
  Thus $f^* \geq h_{u,\alpha} - \varepsilon = h_u - \varepsilon$ on $F_v$.
  Since $\varepsilon$ was arbitrary, $f^* \geq h_u$ on $F_v$, as desired.

  It follows that the antipode identity holds on $\cE_0\bigl( B(E^*) \bigr)$.
  By continuity of $f^*$, it holds throughout $\cE\bigl( B(E^*) \bigr)$.
\end{proof}

Thus isolated types satisfy the antipode identity at some boundary points (and recall that by \autoref{lem:AntipodeIdentityRealisedType}, the antipode identity at a non-boundary point amounts to the type being realised).
When $\dim E = 1$ we recover the characterisation of isolated types given in \autoref{sec:IsolatedTypesDimensionOne}.

\begin{cor}
  Assume $\dim E = 1$ and let $\partial B(E^*) = \{\pm\lambda\}$.
  Then a type $f \in K_C(E)$ is isolated if and only if it satisfies the antipode identity at $\lambda$.
  By \autoref{lem:ConvexKatetovConjugate}\autoref{item:ConvexKatetovConjugateHB}, this is equivalent to: any vector $v \in E$ is smooth in the generated extension $E[f]$.
\end{cor}
\begin{proof}
  If $f$ is isolated then the antipode identity holds at $\lambda$ by \autoref{cor:IsolatedTypeNecessaryCondition}.
  Conversely, if the antipode identity holds at $\pm \lambda$, namely, on the entire boundary, then $f$ is $\partial$-extreme.
  Since, in dimension one, every closed convex function is continuous, $f$ is isolated.
\end{proof}

\section{Existence and non-existence results}
\label{sec:Existence}

This section consists of examples of various cases where densely many isolated types are known to exist or not to exist.
We do \emph{not} have a full characterisation of separable spaces $E$ such that isolated types over $E$ are dense.

\begin{dfn}
  \label{dfn:StronglyBoudaryExtreme}
  Let $E$ be a Banach space.
  We say that $f \in K_C(E)$ is \emph{strongly $\partial$-extreme} if
  \begin{enumerate}
  \item The antipode identity $f^*(\lambda) + f^*(-\lambda) = 0$ holds on $\cE(B(E^*))$.
  \item The values of $f^*$ on $\partial B(E^*)$ are maximal given the values of $f^*$ at the extreme points and the fact that $f^*$ is convex.
  \end{enumerate}
\end{dfn}

Clearly, a strongly $\partial$-extreme type is in particular $\partial$-extreme.

\begin{lem}
  \label{lem:IsolatedTypesFiniteDimension}
  Let $E$ be a normed space.
  \begin{enumerate}
  \item Then the set of strongly $\partial$-extreme $f \in K_C(E)$ is dense.
  \item Assume moreover that $\dim E < \infty$ and $f^*\rest_{\partial B(E^*)}$ is continuous whenever $f \in K_C(E)$ is strongly $\partial$-extreme.
    Then the isolated types over $E$ are dense and $\bG[E]$ exists.
  \end{enumerate}
\end{lem}
\begin{proof}
  In order to show that the strongly $\partial$-extreme types are dense, let $U$ be open and $f \in U$.
  We may assume that $U$ consists of all $g \in K_C(E)$ such that $|g(v_i) - f(v_i)| < \varepsilon$ for $i < n$.
  Let $X = \co(v_i : i < n) \subseteq E$, and replacing $f$ with $\widetilde{f\rest_X}$ we may assume that $f$ is local, i.e., that $f^*$ is continuous.
  For each $i < n$ fix $\lambda_i \in B(E^*)$ such that $f(v_i) + f^*(\lambda_i) < \lambda_i v_i  + \varepsilon$, and we may require $\|\lambda_i\| < 1$.
  Define
  \begin{gather*}
    \hat f(\lambda) =
    \begin{cases}
      \half[f^*(\lambda) - f^*(-\lambda)] & \lambda \in \cE(B(E^*)), \\
      f^*(\lambda_i) & \lambda = \lambda_i.
    \end{cases}
  \end{gather*}
  Then $\hat f$ satisfies the hypotheses of \autoref{lem:GeneratedConvexFunction}, and is therefore the restriction of $g^*\colon B(E^*) \rightarrow \bR$ for some $g \in K_C(E)$.
  We have $g^* \geq f^*$, i.e., $g \leq f$, and $g(v_i) \geq \lambda_i v_i - g^*(\lambda_i) = \lambda_i v_i - f^*(\lambda_i) > f(v_i) - \varepsilon$, so $g \in U$.
  Also, $g^*$ is strongly $\partial$-extreme by construction.

  Assume now that $E$ is as in the second item.
  By \autoref{lem:ConvexBoundaryContinuity}, if $f \in K_C(E)$ is strongly $\partial$-extreme then $f^*$ is continuous, so $f$ is isolated by \autoref{thm:IsolatedTypeCharacterisation}.
  It follows that the isolated types are dense in $K_C(E) = \tS_1(E)$, so $\bG[E]$ exists by \autoref{thm:AtomicGurarijExistence}.
\end{proof}

\begin{rmk}
  \label{rmk:IsolatedTypeReflexive}
  Over a reflexive Banach space $E$, every isolated type is strongly $\partial$-extreme, by \autoref{cor:IsolatedTypeNecessaryCondition} and the fact that every $\lambda \in \partial B(E^*)$ norms some $v \neq 0$.
\end{rmk}

\renewcommand{\labelenumii}{\theenumii}
\renewcommand{\theenumii}{\textnormal{(\alph{enumii})}}

\begin{thm}
  \label{thm:IsolatedTypesFiniteDimension}
  Let $E$ be a normed space of finite dimension.
  Then the isolated types in $K_C(E)$ are dense if either of the following holds:
  \begin{enumerate}
  \item \label{item:IsolatedTypesFiniteDimensionSimplex}
    Every face of $B(E^*)$ of dimension at most $\dim E - 2$ is a simplex.
    This holds in particular whenever $\dim E \leq 3$.
    Special cases of this include:
    \begin{enumerate}
    \item \label{item:IsolatedTypesFiniteDimensionSimplexAll}
      Every proper face of $B(E^*)$ is a simplex.
      This holds in particular whenever $\dim E \leq 2$.
      In this case $\bG$ is atomic over $E \subseteq \bG$ if and only if every $\lambda \in E^*$ admits a unique extension of the same norm to $\bG$.
    \item The space $E$ is smooth, i.e., every $v \in E$ is smooth.
      In this case $\bG$ is atomic over $E \subseteq \bG$ if and only if every $v \in E$ is smooth in $\bG$.
    \end{enumerate}
  \item \label{item:IsolatedTypesFiniteDimensionPolyhedral}
    The space $E$ is polyhedral.
  \end{enumerate}
\end{thm}
\begin{proof}
  In each case we apply \autoref{lem:IsolatedTypesFiniteDimension}, so we assume throughout that $f \in K_C(E)$ is strongly $\partial$-extreme.

  In the first case, let $n = \dim E$ and let $X \subseteq \partial B(E^*)$ be the union of all faces of $\partial B(E^*)$ that are simplexes.
  Then $\partial B(E^*) \setminus X$ consists of the relative interiors of some faces of dimension $n-1$, so $X$ is closed.
  On each face that is a simplex, $f^*$ is affine, and since it satisfies the antipode identity at the extreme points in satisfies it throughout $X$.
  By \autoref{lem:ConvexKatetovConjugate}\autoref{item:ConvexKatetovAntipodeIdentity}, $f^*$, as a function on $B(E^*)$, is continuous at every $\lambda \in X$.
  It follows by \autoref{lem:ConvexBoundaryContinuity} that the restriction of $f^*$ to any face (be it a simplex or not) is continuous.
  Thus, if $\lambda_k \in \partial B(E^*)$, $\lambda_k \rightarrow \lambda$, then either $\lambda$ belongs to the relative interior of some face of dimension $n-1$ or else belongs to $X$, and in any case $f^*(\lambda_k) \rightarrow f^*(\lambda)$.

  In the first special case $X = \partial B(E^*)$ and the characterisation of $\bG[E]$ follows from \autoref{lem:ConvexKatetovConjugate}\autoref{item:ConvexKatetovConjugateHB}.
  The second special case is clear.

  If $E$ is polyhedral, one prove by induction on $m$, using \autoref{lem:ConvexBoundaryContinuity}, that $f^*$ is continuous on each face of dimension $m$ and therefore on the union of all such faces.
  For $m=n-1$, this means that $f^*$ is continuous on $\partial B(E^*)$.
\end{proof}

Notice that \autoref{prp:IsolatedTypesDimensionOne} fits all cases mentioned in \autoref{thm:IsolatedTypesFiniteDimension}.
More generally, the case where $E = \ell_\infty(n)$, namely $\bR^n$ equipped with the supremum norm, fits cases \autoref{item:IsolatedTypesFiniteDimensionSimplexAll} and \autoref{item:IsolatedTypesFiniteDimensionPolyhedral}.
We next use this to show that there are infinitely many distinct orbits in the action of $\Aut(\bG) \curvearrowright \partial B(\bG)$ (previously we only knew there were at least two, since there are both smooth and non-smooth points).

\begin{cor}
  \label{cor:DistinctOrbits}
  For each $n \in \bN$ there exists $v \in \partial B(\bG)$ such that $N(v) = \bigl\{ \lambda \in \partial B(\bG^*) : \lambda v = 1 \bigr\}$ is a simplex of dimension $n$.
  Consequently, the action $\Aut(\bG) \curvearrowright \partial B(\bG)$ admits infinitely many distinct orbits.
\end{cor}
\begin{proof}
  By \autoref{thm:IsolatedTypesFiniteDimension}\autoref{item:IsolatedTypesFiniteDimensionSimplexAll}, $\bG\bigl[\ell_\infty(n+1))\bigr]$ exists, and letting $v = (1,1,\ldots) \in \ell(n+1)$, the set $N(v)$ is a simplex of dimension $n$.

  Say that two convex sets are isomorphic if there exists a homeomorphism between them respecting convex combinations.
  Then the isomorphism type of $N(v)$ is invariant under the action of $\Aut(\bG)$, whence the existence of distinct orbits.
\end{proof}



Let us now give some examples in which the conclusion of \autoref{thm:IsolatedTypesFiniteDimension} fails.
This will show that while \autoref{thm:IsolatedTypesFiniteDimension} is not necessarily optimal, none of the hypotheses can be simply done away with.

\begin{exm}
  \label{exm:IsolatedTypesDimension4}
  We construct an example of a space $E$ of dimension $4$, such that $\bG[E]$ does not exist.
  Let $B_0 = \{\pm 1\}^4 \subseteq \bR^4$, $B_1 = \{ (x,y,0,0) : x^2 + y^2 = 2 \text{ and } x,y \neq \pm 1 \}$ and $B = \co(B_0 \cup B_1)$.
  Then $B$ is a compact symmetric convex neighbourhood of $0$, so we may take $E^* = \bR^4$ with $B(E^*) = B$.
  Moreover, the set of extreme points in $B$ is exactly $B_0 \cup B_1$.
  It follows by \autoref{cor:IsolatedTypeNecessaryCondition} that if $f \in K_C(E)$ is isolated, then $f^*$ satisfies the antipode identity also at the four points $(\pm 1, \pm 1, 0,0)$ (which are not extreme).

  Let us construct a special $f \in K_C(E)$ by constructing $f^*$.
  At a point $(x,y,z,w) \in B_0$ we let $f^*(x,y,z,w) = xyz$, on $B_1$ we let $f^*$ vanish, and the define $f^*$ on $B$ as the generated closed convex function.
  This function satisfies the antipode inequality and therefore is indeed of the form $f^*$ for some $f \in K_C(E)$.
  In addition, a direct calculation reveals that $f^*(\pm 1, \pm 1, 0,0) = -1$.

  Fix $\varepsilon > 0$ ($\varepsilon = 1/2$ will do).
  At each point $\lambda \in B_0$ there is some $v_\lambda \in E$ such that $f^*(\lambda) < \lambda v_\lambda - f(v_\lambda) + \varepsilon$.
  Let $U \subseteq K_C(E)$ consist of all $g$ such that $|g(v_\lambda) - f(v_\lambda)| < \varepsilon$ for all $\lambda \in B_0$.
  Then $U$ is a neighbourhood of $f$, and if $g \in U$, then at $\lambda \in B_0$ we have
  \begin{gather*}
    g^*(\lambda)
    \geq \lambda v_\lambda - g(v_\lambda)
    > \lambda v_\lambda - f(v_\lambda) - \varepsilon
    > f^*(\lambda) - 2\varepsilon,
  \end{gather*}
  and therefore
  \begin{gather*}
    g^*(\lambda)
    \leq -g(-\lambda)
    < -f^*(-\lambda) + 2\varepsilon
    = f^*(\lambda) + 2\varepsilon.
  \end{gather*}
  Thus $g^* < f^* + 2\varepsilon$ throughout the unit cube, and in particular $g^*(\pm 1, \pm 1, 0, 0) < 2\varepsilon - 1$.
  Thus, for $\varepsilon = 1/2$ or less, the antipode identity fails at $(\pm 1, \pm 1,0,0)$ for every $g \in U$, so $U$ contains no isolated points.
\end{exm}

If we want counter-examples consisting of smooth spaces we need to move to infinite dimension.
In fact, we obtain a plethora of examples over which there are no isolated types other than the obvious ones.

\begin{prp}
  \label{prp:IsolatedTypesNone}
  Let $E$ be a Banach space such that $\overline{\cE(B(E^*))} \nsubseteq \partial B(E^*)$.
  Then the only isolated types over $E$ are the realised ones.

  This holds in particular when $E = c_0$, $E^* = \ell_1$ or $E = \ell_p$, $E^* = \ell_q$ with $1 < p,q$, $\frac{1}{p} + \frac{1}{q} = 1$, since $\lambda_i \rightarrow 0$, where $\lambda_i$ consists of a single $1$ at position $i$ and $0$ elsewhere.
\end{prp}
\begin{proof}
  Indeed, let $f \in K_C(E)$ be isolated.
  By \autoref{cor:IsolatedTypeNecessaryCondition}, $f^*$ satisfies the antipode identity on $\cE(B(E^*))$ and therefore (by continuity of $f^*$) at some non-boundary point.
  By \autoref{lem:AntipodeIdentityRealisedType}, $f$ is realised.
\end{proof}

Over $E = \bG$ the isolated types are again exactly the realised ones, but in this specific case they are dense and $\bG[E] = E$.

\begin{qst}
  Is there any infinite-dimensional $E$ other than $\bG$ over which the isolated types are dense?
  Is there any infinite-dimensional $E$ over which there are unrealised isolated types?
  Specifically, what happens in the case $E = \ell_1$, $E^* = \ell_\infty$, to which \autoref{prp:IsolatedTypesNone} does not apply?
\end{qst}

\section{Counting types}
\label{sec:CountingTypes}

We conclude with a calculation of the size of the type-space over a separable Banach space $E$.
By ``size'' we mean here the metric density character (since the cardinal $|\tS_n(E)|$ is the continuum as soon as $n > 0$ and $E \neq 0$).

\begin{thm}
  \label{thm:CountingTypes}
  Let $E$ be a separable Banach space.
  \begin{enumerate}
  \item \label{item:CountingTypesFew}
    If $E$ is finite-dimensional and polyhedral, then $\tS_n(E)$ is metrically separable.
  \item \label{item:CountingTypesMany}
    Otherwise, $\tS_n(E)$ has metric density character equal to the continuum for every $n \geq 1$.
  \end{enumerate}
\end{thm}
\begin{proof}
  Assume first that $E$ is finite-dimensional and polyhedral.
  Then by Melleray \cite[Remarks following Corollary~4.6]{Melleray:UrysohnGeometry}, the space $K(E)$ is separable, and \textit{a fortiori} so is $\tS_1(E) = K_C(E)$.
  The passage from $1$-types to $n$-types is done as in the proof of \autoref{lem:IsolatedTypeOneVariable}, and is left to the reader.

  Now assume that $E$ is not both finite-dimensional and polyhedral.
  Then by Lindenstrauss \cite[Theorem~7.7]{Lindenstrauss:ExtensionOfCompactOperators} there exists a sequence $\{v_n\} \subseteq E$ such that for any $n \neq m$ and choice of signs:
  \begin{gather*}
    \|v_n \pm v_m\| \leq \|v_n\| + \|v_m\| - 1.
  \end{gather*}
  Embed $E$ (isometrically) in $\ell_\infty$, and for a sequence $\bar \varepsilon \in \{\pm 1\}^\bN$, consider the family of closed balls $\overline{B}(\varepsilon_n v_n, \|v_n\| - \half)$.
  By hypothesis every two such balls intersect at a non empty set, and therefore there exists $v \in \ell_\infty$ that belongs to them all.
  In other words, there exists $\xi_{\bar \varepsilon} = \tp(v/E) \in \tS_1(E)$ such that $\|x - \varepsilon_n v_n\|^{\xi_{\bar \varepsilon}} \leq \|v_n\| - \half$.
  If $\bar \varepsilon \neq \bar \varepsilon'$, then $d(\xi_{\bar \varepsilon},\xi_{\bar \varepsilon'}) \geq 1$, so the density character of $\tS_1(E)$ is at least the continuum.
  The same holds \textit{a fortiori} for $\tS_n(E)$, $n \geq 1$.
\end{proof}

\begin{rmk}
  Lindenstrauss's argument is quite elementary and yields a quick proof for \autoref{thm:CountingTypes}\ref{item:CountingTypesMany} that does not depend on the machinery developed in earlier sections.
  An argument closer to the spirit of the present paper can also be given.

  Let $\Xi$ be the set of lower semi-continuous functions $f_0\colon \cE(B(E^*)) \rightarrow \bR$ that satisfy in addition $f_0(\lambda) + f_0(-\lambda) \leq 0$.
  Then $E$ is not a finite-dimensional polyhedral space if and only if $\cE(B(E^*))$ is infinite, in which case $\Xi$ has density character continuum.
  If $f_0 \in \Xi$ and $f = f_0^*$ as in \autoref{lem:GeneratedConvexFunction}, then $f^*\rest_{\cE(B(E^*))} = f_0$ and $f^*(\lambda) + f^*(-\lambda) \leq 0$ throughout $B(E^*)$, so $f \in K_C(E)$ and we are done.
  Notice that this argument has the advantage of treating the two cases of ``finite-dimensional, non polyhedral'' and ``infinite-dimensional'' in the same manner, while the proof of \cite[Theorem~7.7]{Lindenstrauss:ExtensionOfCompactOperators} treats them separately, with the second one being significantly more involved.
\end{rmk}

\autoref{thm:CountingTypes}\autoref{item:CountingTypesMany} answers Problem 2 of Avilés et al.\ \cite[Section~4]{Aviles-CabelloSanchez-Castillo-Gonzalez-Moreno:UniversalDisposition} in the negative (and we thank Wiesław \textsc{Kubiś} for having pointed this out to us).
They say that a Banach space $G$ is \emph{of universal disposition for finite-dimensional spaces} if it satisfies a strengthening of \autoref{dfn:Gurarij} with $\psi$ being an isometry.

\begin{cor}
  The density character of any space of universal disposition for finite-dimensional spaces is at least the continuum.
  In other words, the answer to Problem 2 of \cite[Section~4]{Aviles-CabelloSanchez-Castillo-Gonzalez-Moreno:UniversalDisposition} is negative.
\end{cor}
\begin{proof}
  Assume that $G$ is of universal disposition for finite-dimensional spaces.
  Then the Euclidean plane $E$ embeds isometrically in $G$, and all types over $E$ are realised in $G$, so the density character of $G$ must be at least the metric density character of $S_1(E)$, namely the continuum.
\end{proof}

On the other hand, say that a Gurarij space $G$ is \emph{strongly $\aleph_1$-homogeneous} if the following stronger version of \autoref{cor:HomogeneousGurarij} holds in $G$:
\begin{quote}
  For every separable $F \subseteq G$ and isometric embedding $\varphi\colon F \rightarrow G$ there exists an isometric automorphism $\psi \in \Aut(G)$ extending $\varphi$.
\end{quote}
Clearly, a strongly $\aleph_1$-homogeneous Gurarij space is of universal disposition for finite-dimensional (and even separable) spaces.
Moreover, there does exist such a space of density character the continuum.
This is merely a special case of a general model theoretic result: for any cardinal $\kappa$ and structure $\bM$ of density character $\leq 2^\kappa$, in a language of cardinal $\leq \kappa$, there exists an elementary extension $\bM' \succeq \bM$ of density character still $\leq 2^\kappa$, which is moreover $\kappa^+$-saturated and strongly $\kappa^+$-homogeneous.
Apply this to $\bM = \bG$ and $\kappa = \aleph_0$.

\newcommand{\etalchar}[1]{$^{#1}$}
\providecommand{\bysame}{\leavevmode\hbox to3em{\hrulefill}\thinspace}


\begin{thebibliography}{{Ben}08b}

\bibitem[ACC{\etalchar{+}}11]{Aviles-CabelloSanchez-Castillo-Gonzalez-Moreno:UniversalDisposition}
Antonio \bgroup\scshape{}Avilés\egroup{}, Félix \bgroup\scshape{}{Cabello
  Sánchez}\egroup{}, Jesús M.~F. \bgroup\scshape{}Castillo\egroup{}, Manuel
  \bgroup\scshape{}González\egroup{}, and Yolanda
  \bgroup\scshape{}Moreno\egroup{}, \emph{Banach spaces of universal
  disposition}, Journal of Functional Analysis \textbf{261} (2011), no.~9,
  2347--2361,
  \href{http://dx.doi.org/10.1016/j.jfa.2011.06.011}{doi:10.1016/j.jfa.2011.06.011}.

\bibitem[BBHU08]{BenYaacov-Berenstein-Henson-Usvyatsov:NewtonMS}
Itaï \bgroup\scshape{}{Ben Yaacov}\egroup{}, Alexander
  \bgroup\scshape{}Berenstein\egroup{}, C.~Ward
  \bgroup\scshape{}Henson\egroup{}, and Alexander
  \bgroup\scshape{}Usvyatsov\egroup{},
  \href{http://math.univ-lyon1.fr/~begnac/articles/mtfms.pdf} {\emph{Model
  theory for metric structures}}, Model theory with applications to algebra and
  analysis. {V}ol. 2, London Math. Soc. Lecture Note Ser., vol. 350, Cambridge
  Univ. Press, Cambridge, 2008, pp.~315--427,
  \href{http://dx.doi.org/10.1017/CBO9780511735219.011}{doi:10.1017/CBO9780511735219.011}.

\bibitem[{Ben}08a]{BenYaacov:Unbounded}
Itaï \bgroup\scshape{}{Ben Yaacov}\egroup{},
  \href{http://math.univ-lyon1.fr/~begnac/articles/Unbdd.pdf} {\emph{Continuous
  first order logic for unbounded metric structures}}, Journal of Mathematical
  Logic \textbf{8} (2008), no.~2, 197--223,
  \href{http://dx.doi.org/10.1142/S0219061308000737}{doi:10.1142/S0219061308000737},
  \href{http://arxiv.org/abs/0903.4957}{arXiv:0903.4957}.

\bibitem[{Ben}08b]{BenYaacov:TopometricSpacesAndPerturbations}
\bysame, \href{http://math.univ-lyon1.fr/~begnac/articles/TopoPert.pdf}
  {\emph{Topometric spaces and perturbations of metric structures}}, Logic and
  Analysis \textbf{1} (2008), no.~3--4, 235--272,
  \href{http://dx.doi.org/10.1007/s11813-008-0009-x}{doi:10.1007/s11813-008-0009-x},
  \href{http://arxiv.org/abs/0802.4458}{arXiv:0802.4458}.

\bibitem[{Ben}09]{BenYaacov:NakanoSpaces}
\bysame, \href{http://math.univ-lyon1.fr/~begnac/articles/Nakano.pdf}
  {\emph{Modular functionals and perturbations of {N}akano spaces}}, Journal of
  Logic and Analysis \textbf{1} (2009), Paper 1, 42 pp.,
  \href{http://dx.doi.org/10.4115/jla.2009.1.1}{doi:10.4115/jla.2009.1.1},
  \href{http://arxiv.org/abs/0802.4285}{arXiv:0802.4285}.

\bibitem[{Ben}14]{BenYaacov:UniversalGurarijIsometryGroup}
\bysame, \href{http://math.univ-lyon1.fr/~begnac/articles/Katetov.pdf}
  {\emph{The linear isometry group of the {G}urarij space is universal}},
  Proceedings of the American Mathematical Society \textbf{142} (2014), no.~7,
  2459--2467,
  \href{http://dx.doi.org/10.1090/S0002-9939-2014-11956-3}{doi:10.1090/S0002-9939-2014-11956-3},
  \href{http://arxiv.org/abs/1203.4915}{arXiv:1203.4915}.

\bibitem[{Ben}15]{BenYaacov:MetricFraisse}
\bysame, \href{http://math.univ-lyon1.fr/~begnac/articles/Fraisse.pdf}
  {\emph{Fraïssé limits of metric structures}}, Journal of Symbolic Logic
  \textbf{80} (2015), no.~1, 100--115,
  \href{http://dx.doi.org/10.1017/jsl.2014.71}{doi:10.1017/jsl.2014.71},
  \href{http://arxiv.org/abs/1203.4459}{arXiv:1203.4459}.

\bibitem[Bou81]{Bourbaki:EspacesVectorielsTopologiques}
Nicolas \bgroup\scshape{}Bourbaki\egroup{}, \emph{Espaces vectoriels
  topologiques. {C}hapitres 1 \`a 5}, new ed., Masson, Paris, 1981,
  {\'E}l{\'e}ments de math{\'e}matique.

\bibitem[Bre83]{Brezis:AnalyseFonctionnelle}
Ha{\"{\i}}m \bgroup\scshape{}Brezis\egroup{}, \emph{Analyse fonctionnelle},
  Collection Math\'ematiques Appliqu\'ees pour la Ma\^\i trise, Masson, Paris,
  1983, Th{\'e}orie et applications.

\bibitem[BU07]{BenYaacov-Usvyatsov:dFiniteness}
Itaï \bgroup\scshape{}{Ben Yaacov}\egroup{} and Alexander
  \bgroup\scshape{}Usvyatsov\egroup{},
  \href{http://math.univ-lyon1.fr/~begnac/articles/dfin.pdf} {\emph{On
  $d$-finiteness in continuous structures}}, Fundamenta Mathematicae
  \textbf{194} (2007), 67--88,
  \href{http://dx.doi.org/10.4064/fm194-1-4}{doi:10.4064/fm194-1-4}.

\bibitem[BU10]{BenYaacov-Usvyatsov:CFO}
\bysame, \href{http://math.univ-lyon1.fr/~begnac/articles/cfo.pdf}
  {\emph{Continuous first order logic and local stability}}, Transactions of
  the American Mathematical Society \textbf{362} (2010), no.~10, 5213--5259,
  \href{http://dx.doi.org/10.1090/S0002-9947-10-04837-3}{doi:10.1090/S0002-9947-10-04837-3},
  \href{http://arxiv.org/abs/0801.4303}{arXiv:0801.4303}.

\bibitem[Gur66]{Gurarij:UniversalPlacement}
Vladimir~I. \bgroup\scshape{}Gurarij\egroup{}, \emph{Spaces of universal
  placement, isotropic spaces and a problem of {M}azur on rotations of {B}anach
  spaces}, Akademija Nauk SSSR. Sibirskoe Otdelenie. Sibirski\u\i\ Matemati\v
  ceski\u\i\ \v Zurnal \textbf{7} (1966), 1002--1013.

\bibitem[Kat88]{Katetov:UniversalMetricSpaces}
Miroslav \bgroup\scshape{}Katětov\egroup{}, \emph{On universal metric spaces},
  General topology and its relations to modern analysis and algebra, {VI}
  ({P}rague, 1986), Res. Exp. Math., vol.~16, Heldermann, Berlin, 1988,
  pp.~323--330.

\bibitem[KS13]{Kubis-Solecki:GurariiUniqueness}
Wiesław \bgroup\scshape{}Kubiś\egroup{} and Sławomir
  \bgroup\scshape{}Solecki\egroup{}, \emph{A proof of uniqueness of the
  {G}urari\u\i\ space}, Israel Journal of Mathematics \textbf{195} (2013),
  no.~1, 449--456,
  \href{http://dx.doi.org/10.1007/s11856-012-0134-9}{doi:10.1007/s11856-012-0134-9},
  \href{http://arxiv.org/abs/1110.0903}{arXiv:1110.0903}.

\bibitem[Lin64]{Lindenstrauss:ExtensionOfCompactOperators}
Joram \bgroup\scshape{}Lindenstrauss\egroup{}, \emph{Extension of compact
  operators}, Memoirs of the American Mathematical Society \textbf{48} (1964),
  112.

\bibitem[Lus76]{Lusky:UniqueGurarij}
Wolfgang \bgroup\scshape{}Lusky\egroup{}, \emph{The {G}urarij spaces are
  unique}, Archiv der Mathematik \textbf{27} (1976), no.~6, 627--635.

\bibitem[Maz33]{Mazur:Konvexe}
Stanisław \bgroup\scshape{}Mazur\egroup{}, \emph{Über konvexe mengen in
  linearen normierten räumen}, Studia Mathematica \textbf{4} (1933), 70--84.

\bibitem[Mel07]{Melleray:UrysohnGeometry}
Julien \bgroup\scshape{}Melleray\egroup{}, \emph{On the geometry of {U}rysohn's
  universal metric space}, Topology and its Applications \textbf{154} (2007),
  no.~2, 384--403,
  \href{http://dx.doi.org/10.1016/j.topol.2006.05.005}{doi:10.1016/j.topol.2006.05.005}.

\bibitem[Roc70]{Rockafellar:ConvexAnalysis}
R.~Tyrrell \bgroup\scshape{}Rockafellar\egroup{}, \emph{Convex analysis},
  Princeton Mathematical Series, No. 28, Princeton University Press, Princeton,
  N.J., 1970.

\bibitem[Usp08]{Uspenskij:SubgroupsOfMinimalTopologicalGroups}
Vladimir~V. \bgroup\scshape{}Uspenskij\egroup{}, \emph{On subgroups of minimal
  topological groups}, Topology and its Applications \textbf{155} (2008),
  no.~14, 1580--1606,
  \href{http://dx.doi.org/10.1016/j.topol.2008.03.001}{doi:10.1016/j.topol.2008.03.001},
  \href{http://arxiv.org/abs/math/0004119}{arXiv:math/0004119}.

\end{thebibliography}
\end{document}